\documentclass[a4paper,11pt]{amsart}

\usepackage{amsmath, amsthm, amsfonts, amssymb, enumerate}
\usepackage[bookmarks=false]{hyperref} 
\usepackage{xcolor}
\usepackage{verbatim}
\usepackage[normalem]{ulem}
\usepackage{tikz-cd}

\title{Structural and Non-isomorphism results for $q$-Araki-Woods factors}

\author{Changying Ding}
\address{Department of Mathematics, University of California, Los Angeles, Los Angeles, CA 90095, USA}
\email{cding@math.ucla.edu}

\author{Hui Tan}
\address{Department of Mathematics, University of California, Los Angeles, Los Angeles, CA 90095, USA}
\email{tanhui@math.ucla.edu}

\newtheorem{thm}{Theorem}[section]
\newtheorem{prop}[thm]{Proposition}

\newtheorem{cor}[thm]{Corollary}
\newtheorem{lem}[thm]{Lemma}
\theoremstyle{definition}

\newtheorem{defn/lem}[thm]{Definition/Lemma}
\newtheorem{rem}[thm]{Remark}

\newtheorem{note}[thm]{Notation}

\newcommand{\B}{{\mathbb B}}
\newcommand{\C}{{\mathbb C}}
\newcommand{\F}{{\mathbb F}}

\newcommand{\K}{{\mathbb K}}
\newcommand{\M}{{\mathbb M}}
\newcommand{\N}{{\mathbb N}}

\newcommand{\R}{{\mathbb R}}
\newcommand{\bS}{{\mathbb S}}

\newcommand{\X}{{\mathbb X}}
\newcommand{\Y}{{\mathbb Y}}

\newcommand{\cF}{{\mathcal F}}

\newcommand{\cH}{{\mathcal H}}
\newcommand{\cI}{{\mathcal I}}

\newcommand{\cK}{{\mathcal K}}
\newcommand{\cL}{{\mathcal L}}
\newcommand{\cM}{{\mathcal M}}
\newcommand{\cN}{{\mathcal N}}
\newcommand{\cO}{{\mathcal O}}

\newcommand{\cU}{{\mathcal U}}

\newcommand{\cZ}{{\mathcal Z}}

\newcommand{\sH}{\mathsf{H}}
\newcommand{\sK}{\mathsf{K}}
\newcommand{\sL}{\mathsf{L}}
\newcommand{\sF}{\mathsf{F}}

\newcommand{\Ad}{\operatorname{Ad}}

\newcommand{\Aut}{\operatorname{Aut}}

\newcommand{\id}{\operatorname{id}}

\newcommand{\Tr}{\operatorname{Tr}}

\newcommand{\ot}{\otimes}
\newcommand{\ovt}{\, \bar{\otimes}\,}

\newcommand{\op}{{\rm op}}

\newcommand{\alg}{{\operatorname{alg}}}

\newcommand{\ds}{{\sharp\kern-.5pt\sharp}}

\newcommand{\actson}{{\, \curvearrowright \,}}

\makeatletter
\DeclareRobustCommand\frownotimes{\mathbin{\mathpalette\frown@otimes\relax}}
\newcommand{\frown@otimes}[2]{%
  \vbox{
    \ialign{##\cr
      \hidewidth$\m@th#1{}_\frown$\kern-\scriptspace\hidewidth\cr
      \noalign{\nointerlineskip\kern-1pt}
      $\m@th#1\otimes$\cr
    }%
  }%
}
\makeatother

\begin{document}

\begin{abstract}
It is proved that the $q$-Araki-Woods factor $\Gamma_q(\sH_\R, U)''$ associated with a strongly continuous orthogonal representation 
	$U:\R\to \cO(\sH_\R)$ is strongly solid for all $q\in (-1,1)$
	if the representation $U$ is almost periodic.
We also show that the $q$-Araki-Woods factor $\Gamma_q(\sH_\R, U)''$ 
	is not isomorphic to any free Araki-Woods factor for any $q\in (-1,1)\setminus\{0\}$
	if the representation $U$ has nontrivial weakly mixing part or infinite dimensional almost periodic part with bounded spectrum.
\end{abstract}
\maketitle

\section{Introduction}

Hiai's construction \cite{Hia03} associates a von Neumann algebra $\Gamma_q(\sH_\R, U)''$ called the $q$-Araki-Woods algebra
	to every strongly continuous orthogonal representation of $U:\R\to \cO(\sH_\R)$ on a real Hilbert space $\sH_\R$
	and a parameter $q\in (-1,1)$.
When $q=0$, this is Shlyakhtenko's construction of free Araki-Woods factors \cite{Sh97},
	which is the non-tracial analog of Voiculescu's free Gaussian functor \cite{VDN92}.
On the other hand, if the representation $U$ is trivial, $\Gamma_q(\sH_\R, U)''$
	is the von Neumann algebra of $q$-Gaussian variables of Bo\.zejko and Speicher \cite{BS91},
	which can be seen as a deformed free group factor.
	
In the study of the structure of these algebras, a fundamental result of Ozawa and Popa \cite{OzPo10a}
	asserts that the free group factor $L\F_n$ is strongly solid, i.e.,
	the von Neumann subalgebra generated by the normalizer $\cN_{L\F_n}(A)=\{u\in \cU(L\F_n)\mid uA u^*=A\}$ of
	any diffuse amenable von Neumann subalgebra $A\subset L\F_n$
	remains amenable, which strengthens both Voiculescu's celebrated result on absence of Cartan subalgebra in $L\F_n$ \cite{Voi96} 
	and Ozawa's result on solidity of $L\F_n$ \cite{Oza04}. 
The powerful strategy was later adapted to obtain strong solidity of $q$-Gaussian algebras \cite{Av11}
	(see also \cite{CaIsWa21, DP23})
	and free Araki-Woods factors \cite{BHV18}.
However, the structure of $q$-Araki-Woods algebras is less understood.
For instance, the question regarding factoriality of $\Gamma_q(\sH_\R, U)''$ was only recently resolved in full generality in \cite{KSW23}
	via a conjugate variable approach \cite{MiSp23, Nel17}. 

On the isomorphism side of these algebras, a surprising result of Guionnet and Shlyakhtenko \cite{GuSh14}
	shows that for a finite dimensional Hilbert space $\sH_\R$ 
	and a small range (depending on $\dim(\sH_\R)$) of $q$ around $0$,
	all $q$-Gaussian algebras $\Gamma_q(\sH_\R)''$ are isomorphic to free group factors.
This result was later generalized to the non-tracial setting \cite{Nel15}.
The situation when $\dim(\sH_\R)=\infty$ is quite the opposite, as a result of Caspers shows that $\Gamma_q(\sH_\R)''$ is never isomorphic 
	to any free group factors \cite{Cas23}.

This article continues these two lines of research for $q$-Araki-Woods factors. 
On the structural side, we show in Theorem~\ref{thm: strong solid} that the $q$-Araki-Woods factor is strongly solid
	for any $q\in(-1,1)$ and almost periodic representation $U:\R\to \cO(\sH_\R)$.
A key ingredient is a dichotomy for subalgebras in the continuous core of $\Gamma_q(\sH_\R, U)''$ (Theorem~\ref{thm: dichotomy}),
	which roughly says any von Neumann subalgebra in a finite corner of the continuous core is either amenable
	or is ``properly proximal relative to $L\R$'' in the sense of \cite{DKEP22},
	and can be seen as a continuous core version of the dichotomy for subalgebras in $q$-Gaussian algebras \cite[Theorem 8.1]{DP23}.
Combining this dichotomy with bimodule computations in Section~\ref{sec: bimod} built upon \cite{Av11, Wil20},
	complete metric approximation property of $q$-Araki-Woods \cite{ABW18}
	and the general weak compactness argument from \cite{BHV18}
	yields the strong solidity for almost periodic $q$-Araki-Woods factors.
	
Another consequence of our dichotomy result is a strengthening of fullness of $q$-Araki-Woods factors in the almost periodic case.
It was shown in \cite{KSW23} via \cite{Nel17} that a $q$-Araki-Woods algebra $\Gamma_q(\sH_\R, U)''$ is a full factor
	if $2\leq \dim(\sH_\R)<\infty$
	while the case $U: \R\to \cO(\sH_\R)$ is infinite dimensional and almost periodic was treated separately in \cite{KW24}
	(the weakly mixing case was obtained in \cite{HoIs20}).
With Theorem~\ref{thm: dichotomy}, it turns out that any nonamenable subfactor of almost periodic $q$-Araki-Woods factors
	with expectation is full (Corollary~\ref{cor: fullness}).
	
Coming back to the isomorphism side, we establish a $q$-Araki-Woods analog of Casper's non-isomorphism result \cite{Cas23}
	in Section~\ref{sec: iso} via the notion of biexact von Neumann algebras \cite{DP23}.
Indeed, it is known that all free Araki-Woods algebras are biexact \cite{HoIs17, DP23}
	while a necessary condition for a von Neumann algebra $M$ to be biexact is that the
	$M$-$M$ bimodule $L^2(M^\cU\ominus M)$ is weakly coarse for any non-principle ultrafilter $\cU\in \beta\N\setminus\N$
	\cite[Section 7]{DP23}.
Exploiting an idea from \cite{BCKW22} and using norm estimates from \cite{Nou04, Hia03},
	we prove in Theorem~\ref{thm: non biexact} that
	for any $U: \R\to \cO(\sH_\R)$ that has nontrivial weakly mixing part or infinite dimensional almost periodic part with bounded spectrum
	and $q\in (-1,1)\setminus\{0\}$,
	the $q$-Araki-Woods factor $\Gamma_q(\sH_\R, U)''$ fails the aforementioned necessary condition and hence not isomorphic to any free Araki-Woods factors.
The same approach also allows us to partially resolve \cite[Conjecture 2.11]{KSW23},
	which in particular implies that almost periodic $q$-Araki-Woods factors cannot be classified by Connes' Sd invariant,
	a clear contrast to the free case \cite{Sh97}.

The difference between Shlyakhtenko's functor and Hiai's functor can also be seen in the case of finite dimensional representations.
Indeed, if we denote by $U_\lambda: \R\to\cO(\R^2)$ the rotation with period $2\pi/\log(\lambda)$,
	then a consequence of \cite{Sh97} is that for $q=0$
		the natural inclusions of $\Gamma_q(\oplus_{i=1}^m(\R^2, U_\lambda))''\subset \Gamma_q(\oplus_{i=1}^k (\R^2, U_\lambda))''$
		and $\Gamma_q(\oplus_{i=1}^n(\R^2, U_\lambda))'' \subset \Gamma_q(\oplus_{i=1}^k (\R^2, U_\lambda))''$
		are isomorphic for any natural numbers $m,n< k$.
However, for $q\neq 0$, a consequence of Theorem~\ref{thm: non-iso inclusions} shows that
	one may find $m,n,k\in \N$ (depending on $q$)
	such that these inclusions are not isomorphic,
	and Theorem~\ref{thm: non-iso inclusions} is even new for $q$-Gaussian algebras.
	
\textbf{Acknowledgement.} It is our pleasure to thank Dima Shlyakhtenko for
	his lectures during the Spring 2025 quarter which motivated this project, stimulating discussions and encouragements. 
    We would also like to thank Cyril Houdayer, Brent Nelson and Zhiyuan Yang for their helpful comments.
	The authors are supported by AMS-Simons travel grants.

\section{Preliminaries and notations}\label{sec: prelim}

We first recall some general von Neumann algebra theory and set some notations. See \cite{Take03} for detailed treatment.

Let $M$ be a von Neumann algebra with a faithful normal semifinite weight $\varphi$.
Denote by $L^2(M,\varphi)$ the corresponding Hilbert space and $S_\varphi$ the Tomita operator
	with its polar decomposition $S_\varphi=J_\varphi \Delta_\varphi^{1/2}$.
One has a natural $M$-bimodule structure on $L^2(M,\varphi)$ given by $x\cdot \xi \cdot y=x J_\varphi y^* J_\varphi \xi$ for $x,y\in M$
	and $\xi\in L^2(M,\varphi)$.
If $\psi$ is another faithful normal semifinite weight on $M$,
	then by the uniqueness of the standard form there exists a unitary
	$U_{\varphi,\psi}: L^2(M,\psi)\to L^2(M,\varphi)$
	that intertwines the left and right actions of $M$
	and thus we may identify $L^2(M,\psi)$ with $L^2(M,\varphi)$ as $M$-$M$ bimodules.
	
Let $\sigma^\varphi_t(x)=\Delta_\varphi^{it}x \Delta_{\varphi}^{-it}$ for $t\in \R$ be the modular automorphism group of $\varphi$.
The continuous core of $M$, $M\rtimes_{\sigma^\varphi}\R=c_\varphi(M)$, is the von Neumann algebra generated by
$$
\{\lambda_\varphi(t):=\lambda(t)\ot 1, \pi_\varphi(x)\mid t\in \R, x\in M\}\subset \B(L^2\R\ot L^2(M,\varphi)),
$$
where $\lambda(t)$ is the left regular representation of $\R$ on $L^2\R$,
	and $(\pi_\varphi(x)\xi)(s)=\sigma_{-s}^\varphi(x)\xi(s)$
	for $x\in M$ and $\xi\in L^2(\R, L^2(M,\varphi))=L^2\R\ot L^2(M,\varphi)$.
The Hilbert space $L^2\R\ot L^2(M,\varphi)$ also admits a right $c_\varphi(M)$-module structure given by
$(\xi\cdot x)(t)=\xi(t)\cdot x$ and $(\xi\cdot \lambda_\varphi(s))(t)=\Delta_\varphi^{-is}\xi(t-s)$
	for $\xi\in L^2\R\ot L^2(M,\varphi)$, $x\in M$ and $s,t\in \R$.

If we denote by $\tilde \varphi$ the dual weight to $\varphi$ on $c_\varphi(M)$,	
	the natural $c_\varphi(M)$-bimodule $L^2(c_\varphi(M),\tilde \varphi)$ may be identified with $L^2\R\ot L^2(M,\varphi)$
	by identifying $\lambda_\varphi(f)x$ with $f\ot x\in L^2\R\ot L^2(M,\varphi)$
	for $f\in C_c(\R)$ and $x\in M$ with $\varphi(x^*x)<\infty$ \cite[Section X.1]{Take03}.
Let $h$ be the nonsingular positive self-adjoint operator affiliated with $L_\varphi(\R)=\{\lambda_\varphi(t)\mid t\in \R\}''\subset c_\varphi(M)$
	such that $\lambda_\varphi(t)=h^{it}$, then $\Tr_\varphi=\tilde \varphi(h^{-1/2}\cdot h^{-1/2})$
	is a semifinite faithful normal tracial weight on $c_\varphi(M)$
	and thus we may identify  the $c_\varphi(M)$-bimodule $L^2(c_\varphi(M),\Tr_\varphi)$
	with $L^2(c_\varphi(M),\tilde \varphi)$ via $U_{\tilde\varphi,\Tr_\varphi}$, which is then further identified with $L^2\R\ot L^2(M,\varphi)$.

Suppose $N\subset M$ is a von Neumann subalgebra with a $\varphi$-preserving expectation,
	then $N$ is $\sigma^\varphi$-invariant and we further have $N\rtimes_{\sigma^\varphi}\R\subset M\rtimes_{\sigma^\varphi}\R$,
	which admits a conditional expectation that preserves both $\tilde \varphi$ and $\Tr_\varphi$.
It follows from \cite[Lemma 1.5]{HJKEN} that $U_{\tilde \varphi, \Tr_{\varphi}}$ takes the inclusion $L^2(c_\varphi(N),\Tr_\varphi)\subset  L^2(c_\varphi(M),\Tr_\varphi)$ to $L^2(c_\varphi(N),\tilde \varphi)\subset L^2(c_\varphi(M),\tilde \varphi)$,
	which is further identified with $L^2\R\ot L^2(N,\varphi)\subset L^2\R\ot L^2(M,\varphi)$.
Therefore we may identify the $c_\varphi(N)$-bimodule $L^2(c_\varphi(M)\ominus c_\varphi(N),\Tr_\varphi)$
	with $L^2\R\ot L^2(M\ominus N,\varphi)$.
When $\varphi$ is a state, we may take $N=\C$.

The continuous core of $M$ is independent of the choice of weight $\varphi$. If $\psi$ is another faithful normal semifinite weight on $M$,
	then one has a $*$-isomorphism $\Pi_{\psi,\varphi}: M\rtimes_{\sigma^\varphi}\R\to M\rtimes_{\sigma^\psi}\R$
	that is trace-preserving and restricts to the identity map on $M$.

\subsection{$q$-Araki-Woods factors}
Throughout this paper, we assume all orthogonal representations of $\R$ are strongly continuous representations on separable Hilbert spaces
	and all inner products are conjugate linear in the first variable.

For an orthogonal representation $U:\R\to \cO(\sH_\R)$,
	we reserve the notation $\sH_\C$ for $\sH_\R +i \sH_\R$, the complexification of $\sH_\R$, and
	denote by $\sH$ the closure of $\sH_\C$ under $\langle\cdot,\cdot\rangle_U$,
	where $\langle \xi,\eta\rangle_U=\langle(2A/1+A)\xi,\eta\rangle$
	and $A$ is the generator of the unitary representation $U:\R\to \cU(\sH_\C)$
	following \cite{Sh97}.

Given $q\in (-1,1)$, denote by $\cF_q(\sH)$ the $q$-Fock space of $\sH$, which is the completion of
	$\C\Omega\oplus(\oplus_{n\in \N} \sH^{\ot n})$ under
	$\langle \xi_1\ot \cdots \ot \xi_n, \eta_1\ot \cdots\ot \eta_m\rangle_q
	=\delta_{n,m}\sum_{\sigma\in S_n} q^{i(\sigma)} \prod_{k=1}^n \langle \xi_k, \eta_{\sigma(k)}\rangle_U$
	for $\xi_k, \eta_j\in \sH$,
	where $i(\sigma)$ is the inversion number of the permutation $\sigma\in S_n$.
Given any contraction $T\in \B(\sH)$, we denote by $\cF_q(T)$ the corresponding contraction on $\B(\cF_q(\sH))$
	that satisfies
	$\cF_q(T)(\xi_1\ot \cdots \xi_n)=T\xi_1\ot 	\cdots \ot T\xi_n$ for $\xi_i\in \sH$
	\cite[Lemma 1.4]{BKS97}.

For each $\xi\in \sH$, the $q$-creation operator $\ell(\xi)\in \B(\cF_q(\sH))$ is given by 
	$\ell(\xi)\Omega=\xi$ and $\ell(\xi)(\xi_1\ot\cdots\xi_n)=\xi\ot \xi_1\ot \cdots\ot \xi_n$,
	for any $\{\xi_k\}_{k=1}^n\subset \sH$.
The $q$-Araki-Woods algebra associated with $U:\R\to \cO(\sH_\R)$, introduced in \cite{Hia03},
	is defined by $\Gamma_q(\sH_\R, U)''=\{\ell(\xi)+\ell(\xi)^*\mid \xi\in \sH_\R\}''\subset \B(\cF_q(\sH))$.
It was recently showed in \cite{KSW23} that $\Gamma_q(\sH_\R, U)''$ is a factor whenever $\dim(\sH_\R)\geq 2$
	and hence we will refer to it as the $q$-Araki-Woods factor.
	
For any $\{\xi_i\}_{i=1}^n\subset \sH_\C$, there exists a unique element called the Wick operator 
	$W(\xi_1\ot \cdots \ot \xi_n)\in \Gamma_q(\sH_\R, U)''$
	given by the Wick formula (e.g.\ see \cite[Proposition 2.1]{BMRW23}) such that $W(\xi_1\ot \cdots \ot \xi_n)\Omega=\xi_1\ot \cdots \ot \xi_n$.
Note that the linear span of such operators is strongly dense in $\Gamma_q(\sH_\R, U)''$.
Similarly, for $\{\eta_i\}_{i=1}^n\subset \sH_\R'+i\sH_\R'$, one has a unique element called the right Wick operator
	$W_r(\eta_1\ot \cdots \ot \eta_n)\in \Gamma_q(\sH_\R, U)'$ such that $W_r(\eta_1\ot \cdots \ot \eta_n)\Omega=\eta_1\ot \cdots \ot \eta_n$
	and is given by the right Wick formula,
	where $\sH_\R'=\langle \xi\in \sH\mid \langle \xi, \eta\rangle_U\in\R\  \forall\eta\in \sH_\R\}$.

The vacuum state $\chi_U=\langle \Omega, \cdot \Omega\rangle_q	$ on $\Gamma_q(\sH_\R, U)''$ is called the $q$-quasi-free state
	and is a faithful normal state. 
	When there is no ambiguity, we will use $\chi$ instead of $\chi_U$ for brevity.
The corresponding modular automorphism is given by
	$\sigma^{\chi_U}_t(W(\xi))=W(U_t\xi)$ for $\xi\in \sH_\C$ and $t\in \R$.

\subsection{Ultraproduct von Neumann algebras}

Let $M$ be a $\sigma$-finite von Neumann algebra and $\cU\in \beta\N\setminus \N$ a non-principle ultrafilter.
Define
\[\begin{aligned}
	\cI_\cU(M)&=\{(x_n)_n\in M\mid x_n\to 0 *-{\rm strongly\ as\ }n\to \cU\},\\
	\cM^\cU(M)&=\{(x_n)_n\in M\mid (x_n)_n\cI_\cU(M)\subset \cI_\cU(M){\rm \ and\ }\cI_\cU(M)(x_n)_n\subset \cI_{\cU}(M)\}.
\end{aligned}
\]
The Ocneanu ultraproduct $M^\cU$ is the quotient $\cM^\cU(M)/\cI_\cU(M)$ \cite{Oc85}
	and we denote by $(x_n)^\cU$ the image of $(x_n)_n\in \cM^\cU(M)$ in $M^\cU$.
There is a canonical faithful normal expectation $E_M: M^\cU\to M$ given by $E_M((x_n)^\cU)=\lim_{n\to \cU} x_n$,
	where the limit is in weak$^*$.
If $\varphi$ is a faithful normal state on $M$, then $\varphi^\cU:=\varphi\circ E_M$ gives a faithful normal state on $M^\cU$.
See \cite{AnHa14} for a detailed treatment on ultraproducts.

\subsection{Bimodules}
Given von Neumann algebras $M$ and $N$, a Hilbert space $\cH$ is an $M$-$N$ bimodule if there is a
	$*$-representation $\pi_\cH:M\ot_{\rm alg} N^\op\to \B(\cH)$ such that $\pi$ is normal when restricted to $M\ot \C$
	and $\C\ot N^\op$. 
We use the notation ${_M}\cH{_N}$ to denote an $M$-$N$ bimodule $\cH$.

For $M$-$N$ bimodules ${_M}\cH{_N}$ and ${_M}\cK{_N}$, we say $\cH$ is weakly contained in $\cK$, denoted by ${_M}\cH{_N}\prec {_M}\cK{_N}$
	if $\|\pi_\cH(x)\|\leq \|\pi_\cK(x)\|$ for any $x\in M\ot_{\rm alg} N^\op$.
The standard form of $M$, $L^2(M)$, is an $M$-$M$ bimodule via $\pi_{L^2(M)}(a\ot b^\op)\xi=a Jb^*J \xi$,
	where $J$ is the modular conjugation,
	and similarly we may view $L^2(M)\ot L^2(N)$ as an $M$-$N$ bimodule.
	
An $M$-$N$ bimodule $\cH$ is weakly coarse if ${_M}\cH{_N}\prec {_M} L^2(M)\ot L^2(N){_N}$.
Observe that an $M$-$N$ bimodule is weakly coarse if and only if
	the set of vectors	$\xi\in \cH$ satisfying $M\ot_{\rm alg} N^\op\ni x\mapsto \langle \xi, \pi_\cH(x)\xi\rangle\in \C$ is min-continuous
	generates $\cH$ as an $M$-$N$ bimodule.
	
Let $M$ and $N$ be equipped with faithful normal states $\varphi$ and $\psi$, respectively. 
A vector $\xi$ in an $M$-$N$ bimodule $\cH$ is left $\psi$-bounded if
$$
L_\psi(\xi): N^\op \psi^{1/2}\ni (Ja^*J \psi^{1/2})\mapsto \pi_{\cH}(1\ot a^\op)\xi\in \cH
$$
extends to a bounded map on $L^2(N,\psi)$.
Similarly, a vector $\xi\in\cH$ is right $\varphi$-bounded if
	$R_\varphi(\xi):a\varphi^{1/2} \mapsto \pi_\cH(a\ot 1)\xi$ is bounded on $L^2(M,\varphi)$.
We refer the reader to \cite{Po86, Take03} for comprehensive treatments.

\subsection{Biexact and properly proximal von Neumann algebras}\label{sec: biexact and pp}

The notion of biexact groups was introduced in the seminal paper of Ozawa \cite{Oza04} (see also \cite{BO08})
	and its von Neumann algebra counterpart was introduced in \cite{DP23}.
Since we do not need the actual definition of biexact von Neumann algebras, we only list a few relevant properties for later use.

\begin{lem}\label{lem: biexact prelim}
The following statements are true.	
\begin{enumerate}
\item If a von Neumann algebra $M$ with a faithful normal state $\varphi$ is biexact, 
	then one has $L^2(M^\cU\ominus M, \varphi^\cU)$ is a weakly coarse $M$-$M$ bimodule for any non-principle ultrafilter $\cU\in \beta\N\setminus\N$. \label{item: biexact bimodule character}
\item All free Araki-Woods factors are biexact. \label{item: free araki-woods biexact}
\item Any $q$-Araki-Woods factor associated with a finite dimensional representation is biexact if the ${\rm C}^*$-algebra generated by $\{\ell(e_i)\}_{i=1}^n$ is nuclear, where $\{e_i\}_{i=1}^n$ is a basis of the Hilbert space of the orthogonal representation.
	\label{item: q-araki-woods biexact}
\end{enumerate}
\end{lem}
\begin{proof}
Item~\ref{item: biexact bimodule character} is due to \cite[Theorem 7.19, 7.20]{DP23}
	and item~\ref{item: free araki-woods biexact} is due to \cite[Theorem C.2]{HoIs17} and \cite[Theorem 7.17]{DP23}.
For item~\ref{item: q-araki-woods biexact}, notice that the computation in \cite[Lemma 3.1]{Shl04}
	shows that $\{\ell (e_i)+\ell (e_i)^*\mid i=1,\dots, n\}''$ satisfies strong condition (AO) \cite[Definition 2.6]{HoIs17}
	and hence biexact \cite[Theorem 7.17]{DP23}.
The assumption that ${\rm C}^*(\ell(e_1), \dots, \ell(e_n))$ is nuclear is verified by \cite{JSW94} for $|q|<\sqrt 2-1$
	and by \cite{Ku22} for $|q|<1$.
\end{proof}

We recall the small-at-infinity boundary for a von Neumann algebra in the tracial setting from \cite{DKEP22}.
See \cite{DP23} for its definition in the general setting.

Let $M$ be a tracial von Neumann algebra.
An $M$-boundary piece $\X$ is a hereditary ${\rm C}^*$-subalgebra $\X\subset\B(L^2M)$
	such that $M\cap M(\X)\subset M$ and $JMJ\cap M(\X)\subset JMJ$ are weakly dense,
	and $\X\neq \{0\}$,
	where $M(\X)$ denotes the multiplier algebra of $\X$.
For convenience, we will always assume $\X\neq \{0\}$.
Given an $M$-boundary piece $\X$, define $\K_\X^L(M)\subset \B(L^2M)$ to be the $\|\cdot\|_{\infty,2}$ closure of $\B(L^2M)\X$,
	where $\|T\|_{\infty,2}=\sup_{a\in (M)_1}\|T\hat a\|$
		and $(M)_1=\{a\in M\mid \|a\|\leq 1\}$.
Set $\K_\X(M)=\K_\X^L(M)^*\cap \K_\X^L(M)$,
	then $\K_\X(M)$ is a ${\rm C}^*$-subalgebra that contains $M$ and $JMJ$ in its multiplier algebra \cite[Proposition 3.5]{DKEP22}. 
Put $\K^{\infty,1}_\X(M)=\overline{\K_\X(M)}^{_{\|\cdot\|_{\infty,1}}}\subset \B(L^2M)$, 
	where $\|T\|_{\infty,1}=\sup_{a,b\in (M)_1}\langle T\hat a, \hat b\rangle$,
	and the small-at-infinity boundary for $M$ relative to $\X$ is given by 
$$
\bS_\X(M)=\{T\in\B(L^2M)\mid [T,x]\in \K_\X^{\infty,1}(M),{\rm\ for\ any\ }x\in M'\}.
$$
When $\X=\K(L^2M)$, we omit $\X$ in the above notations.
Given a von Neumann subalgebra $P\subset M$, recall from \cite[Lemma 6.12]{DP23}
	that the $M$-boundary piece $\X$ associated with $P$ is the
	hereditary ${\rm C}^*$-subalgebra of $\B(L^2M)$ 
	generated by $\{x JyJ e_P\mid\  x,y\in M\}$,
	where $e_P: L^2M\to L^2P$ is the orthogonal projection.

Next we recall the notion of properly proximal von Neumann algebra from \cite{DKEP22}. This notion was first introduced for groups
	in \cite{BIP21}.

Let $(M,\tau)$ be a tracial von Neumann algebra and $N\subset M$ a (possibly nonunital) von Neumann subalgebra with expectation.
Given an $M$-boundary piece $\X\subset \B(L^2M)$, we denote by $\X^N=\overline{e_N \K_X(M) e_N}\subset \B(L^2N)$ the $N$-boundary piece
	associated with $\X$ (see \cite[Remark 6.3]{DKEP22}),
	where $e_N: L^2M\to L^2N$ is induced by the normal expectation from $M$ to $N$.
	
We say $M$ is properly proximal relative to $\X$
	if there is no nonzero central projection $z\in \cZ(M)$ and $Mz$-central state
	$\varphi: \bS_{\X}(M)\to \C$ such that $\varphi_{\mid Mz}=\tau_{\mid Mz}$,
	which is equivalent to the condition that
	there is no nonzero projection $z\in \cZ(M)$ and $Mz$-central state $\varphi: \widetilde \bS_\X(M)\to \C$
	such that $\varphi_{\mid Mz}=\tau_{\mid Mz}$
	by the same proof of \cite[Lemma 8.5]{DKEP22}.
Here, the notation $\widetilde \bS_\X(M)$ denotes
$$
\widetilde \bS_\X(M)=\{T\in (\B(L^2M)^\sharp_J)^*\mid [T,JaJ]\in (\K_\X(M)^\sharp_J)^*,{\rm\  for\ all\ }a\in M\},
$$
	$\B(L^2M)^\sharp_J\subset \B(L^2M)^*$ contains all functionals satisfying
	$M\ot_{\rm alg} M\ni a\ot b\mapsto \varphi(aTb)$
	and $JMJ\ot_{\rm alg} JMJ\ni a\ot b\mapsto \varphi(aTb)$ are binormal for any $T\in \B(L^2M)$,
	and $\K_\X(M)^\sharp_J$ is defined similarly.

We point out that the same proof of \cite[Lemma 3.2]{Din24} shows that there exists a u.c.p.\ map
	$$
	\widetilde E: \widetilde \bS_\X(M)\to \widetilde \bS_{\X^N}(N)
	$$	
	such that $\widetilde E_{\mid M}$ is the conditional expectation $E:M\to N$.
Thus if $N$ is not properly proximal relative to $\X^N$,
	then there exists a nonzero central projection $z\in \cZ(N)$
	and a $zN$-central state $\varphi:\widetilde\bS_\X(M)\to \C$ such taht $\varphi_{\mid zMz}=\tau_{zMz}$.	

\section{Weakly coarse bimodules for almost periodic $q$-Araki-Woods factors}\label{sec: bimod}

The aim of this section is to show certain bimodules for $q$-Araki-Woods factors associated with finite dimensional representations
	are weakly coarse, by building on ideas from \cite{Av11, Wil20} in the tracial case.
More precisely, we prove the following.

\begin{prop}\label{prop: coarse step lower bound}
Let $q\in (-1,1)\setminus\{0\}$ and 
	$U: \R\to \cO(\sL_\R)$ be an almost periodic representation with subrepresentations $\sK_\R\subset \sH_\R\subset \sL_\R$,
	where $2\leq \dim(\sK_\R)<\infty$.
Set $N=\Gamma_q(\sK_\R, U)''$, 
	$M=\Gamma_q(\sH_\R, U)''$, $\tilde M=\Gamma_q(\sL_\R, U)''$,
	view $N\subset M\subset \tilde M$ via inclusions of representations 
	and denote by $c(N)$, $c(M)$ and $c(\tilde M)$ the corresponding continuous cores, respectively.
	
Then there exists a constant $\kappa\in \N$ depending on $\sK_\R$ and $q$ such that 
$L^2(p c(\tilde M) p \ominus p c(M) p)^{\ot_{pc(M)p}^k}$ is weakly coarse as a $p c(N)p$-$p c(N)p$ bimodule,
	for any finite trace projection $p\in L_\chi\R$
	and any $k\geq \kappa$,
	where $\chi$ is the $q$-quasi-free state on $\tilde M$.
\end{prop}

We fix the following notations throughout this section.
\begin{note}\label{note 1}
Let $q\in (-1,1)\setminus\{0\}$ be fixed.
Denote by $U: \R \to \cO(\sL_\R)$ an almost periodic representation and $\sH_\R\subset \sL_\R$ a subrepresentation.
Set $M=\Gamma_q(\sH_\R, U)''$, $\tilde M=\Gamma_q(\sL_\R, U)''$
	and denote by $\chi$ the $q$-quasi free state on $\tilde M$, which restricts to the $q$-quasi free state on $M$
	by viewing $M\subset \tilde M$ in the natural way.
Put $c_\chi(\tilde M)=\tilde M\rtimes_{\sigma^\chi}\R$ with semifinite trace $\Tr_\chi$,
	which also restricts to a semifinite trace on $c_\chi(M)=M\rtimes_{\sigma^\chi}\R$.

Consider $\sH_\R'=\sL_\R\ominus \sH_\R$, a subrepresentation of $\sL_\R$, and it follows that $\sL=\sH\oplus \sH'$.
For each $k\in \N$, we set $\cL_k\subset \cF_q(\sH\oplus \sH')=\cF_q(\sL)$ the closed subspace spanned by vectors
	$\xi_1\ot\xi_2\ot\cdots \ot\xi_n\in (\sH\oplus \sH')^{\ot n}$ with $k$ tensor legs in $0\oplus (\sH'\setminus\{0\})$ and the rest in $\sH\oplus 0$,
	and it follows that $\cF_q(\sH\oplus \sH')=\oplus_{k\geq 0}\cL_k$.	
\end{note}

\begin{rem}\label{rem: bimod identify}
By the discussion at the beginning of Section~\ref{sec: prelim}, we have an identification of $c_\chi(\tilde M)$-bimodules
	$\iota: L^2(c_\chi(\tilde M), \Tr_\chi)\to L^2\R\ot L^2(\tilde M,\chi)$
	which restricts to an identification of $c_\chi(M)$-bimodules
	$\iota: L^2(c_\chi(M),\Tr_\chi)\to L^2\R\ot L^2(M,\chi)$ as $M\subset \tilde M$ is with $\chi$-preserving expectation.
Similarly, $\iota$ maps $L^2(L_\chi(\R),\Tr_\chi)$ to $L^2\R\ot \Omega$.
Through this identification bimodules, we may view
$$
L^2(c_\chi(\tilde M) \ominus c_\chi(M),\Tr_\chi)=L^2\R\ot L^2(\tilde M\ominus M, \chi)=L^2\R\ot (\oplus_{k\geq 1} \cL_k),
$$
as $L^2(M,\chi)=\cF_q(\sH\oplus 0)\subset L^2(\tilde M,\chi)=\cF_q(\sH\oplus \sH')$.
Also note that $L^2\R\ot \cL_k\subset L^2\R\ot \cF_q(\sH\oplus \sH')$ is a $c_\chi(M)$-$c_\chi(M)$ sub-bimodule for each $k\in \N$.

If $p\in L_\chi(\R)$ is a finite trace projection, we also identify the following $pc_\chi(M)p$-bimodules
$$
L^2(pc_\chi(\tilde M)p, \Tr_\chi)=pJpJ L^2(c_\chi(\tilde M),\Tr_\chi)=pJpJ (L^2\R\ot \cF_q(\sH\oplus\sH')).
$$

\end{rem}

\begin{lem}\label{lem: inequality for wk coarse}
With Notation~\ref{note 1},
	suppose $\sK_\R\subset \sH_\R$ and $\sF_\R\subset \sL_\R$ 
	are finite dimensional subrepresentations.

For $m, n, i, j, k\in \N$, let 
	$\xi\in \sF^{\otimes m}\cap \cL_k$, $\eta\in \sF^{\ot n}\cap \cL_k$, $\zeta_1\in (\sK\oplus 0)^{\otimes i}$ and $\zeta_2\in (\sK\oplus 0)^{\ot j}$.
Then one has 
\begin{equation}\label{eq: fin dim 1}
|\langle \zeta_2, W(\eta)^* W(\zeta_1) W(\xi)\Omega\rangle_q|\leq C |q|^{ki}\|\zeta_1\|_q\|\zeta_2\|_q,
\end{equation}
where $C$ is a constant depending on $q$, $\xi$, $\eta$ and $U:\R\to\cO(\sK_\R)$.

Moreover, for any $f,g\in C_c(\R)$, one has 
$$|\langle \eta\ot g, W(\zeta_1) JW(\zeta_2)^*J (\xi\ot f)\rangle|\leq C' |q|^{ki}\|\zeta_1\|_q\|\zeta\|_q,$$
where $C'$ is a constant depending on $q$, $\sK_\R$, $\xi\ot f$ and $\eta\ot g$.
\end{lem}
\begin{proof}
Since $\sK_\R$ and $\sF_\R$ are finite dimensional,
	we may find another finite dimensional subspace $\tilde \sK_\R\subset \sL_\R$ containing $\sK_\R$ such that
	$\xi\in \tilde \sK^{\ot m}$ and $\eta\in \tilde \sK^{\ot n}$
	and thus (\ref{eq: fin dim 1}) only concerns $\Gamma_q(\tilde\sK_\R, U)''$,
	for which the proof follows almost the exact same argument as in the tracial case \cite[Proposition 6.6]{Wil20}.
Since the inequality (\ref{eq: fin dim 1}) is rather important for our subsequent uses, we briefly sketch the argument 
	and point out the necessary adjustments.
	
The starting point for proving 
	(\ref{eq: fin dim 1}) is the following formula for products of Wick words from \cite[Theorem 3.3]{EfPo03}
	(see also \cite[Lemma 4.3]{CaIsWa21}),
\begin{equation}\label{eq: intermediate 1}
	W(\Xi_1)W(\Xi_2)W(\Xi_3)=
\sum_{\pi} q^{cr(\pi)} (\prod_{(\ell,r)\in P(\pi)} \langle S\xi_{\ell}, \xi_r\rangle)W(\Xi_{S(\pi)}),
\end{equation}

where $\Xi_1=\ot_{i=1}^{n_1}\xi_i$, $\Xi_2=\ot_{i=n_1+1}^{n_1+n_2}\xi_i$, $\Xi_3=\ot_{i=n_1+n_2+1}^{n_1+n_2+n_3}\xi_i$
	with $\xi_i\in \tilde \sK_\C$,
	the summation is over all partitions $\pi$ of the set $\{1,\dots, n_1+n_2+n_3\}$ 
	that are in $P^{\leq 2}(n_1\ot n_2\ot n_3)$ (see \cite[Definition 4.2]{CaIsWa21}),
	$P(\pi)$ denotes all pairs in $\pi$, $S(\pi)$ denotes all singletons in $\pi$,
	$cr(\pi)$ is the crossing number of $\pi$ and
	$\Xi_{S(\pi)}=\xi_{s_1}\ot \cdots \ot \xi_{s_k}$ if $S(\pi)=\{s_1<\cdots<s_k\}$.
	
We point out that the proof for \cite[Theorem 3.3]{EfPo03} is combinatorial and 
	the only change in the setting of $q$-Araki-Woods is that
	the operator $S: \tilde \sK_\R+i\tilde \sK_\R\ni \xi+i\eta\mapsto \xi-i\eta\in \tilde \sK_\C$ involved in the Wick formula
	is no longer an isometry,
	while the rest of the proof carries out verbatim. 
	
The next step is to derive a version of \cite[Proposition 4.9]{CaIsWa21} based on (\ref{eq: intermediate 1}), which states that
\begin{equation}\label{eq: intermediate 2}
	W(\Xi_1)W(\Xi_2)W(\Xi_3)=\sum_{j,r,s}q^{r(n_2-j-s)} m_r^{13} m_s^{12} m_j^{23}(R^*_{n_1-r-s, r,s}(\Xi_1) 
		R^*_{s, n_2-s-j,j}(\Xi_2) R^*_{r, j, n_3-r-j}(\Xi_3)),
\end{equation}
where $R^*_{i,j,k}$ is the adjoint of the map $R_{i,j,k}:\tilde \sK_q^{\ot i}\ot \tilde \sK_q^{\ot j}\ot \tilde \sK_q^{\ot k}\ni \xi\ot \eta\ot \zeta \mapsto \xi\ot \eta\ot \zeta\in \tilde \sK_q^{\ot i+j+k}$,
	and $m_j: \tilde \sK_q^{\ot j}\ot \tilde \sK_q^{\ot j}\ni \xi\ot \eta\mapsto \langle S\xi, \eta\rangle$.
We remark that since $\tilde \sK_\R$ is finite dimensional, the conjugation operator $S$ is defined on $\tilde\sK$
	and hence one has 
	$S: \tilde \sK^{\ot j}\ni \xi_1\ot \cdots \xi_j\mapsto S\xi_j\ot \cdots \ot S\xi_1\in \tilde \sK^{\ot j}$ as well.

As in the previous step, the exact same argument of \cite[Proposition 4.9]{CaIsWa21} gives us (\ref{eq: intermediate 2}),
	with the only difference being the operator $m_j$ depends on the representation $U: \R\to \cO(\tilde \sK_\R)$.
	
Finally, to obtain (\ref{eq: fin dim 1}) we follow the proof of \cite[Proposition 6.6]{Wil20} which uses (\ref{eq: intermediate 2}).
	The only difference in our setting is that the norm of $m_j$ is bounded by $(\|S_{\mid \tilde\sK}\| d^{1/2})^j$,
	where $d=\dim(\tilde \sK)$, as oppose to $d^{j/2}$ as in \cite[Proposition 6.6]{Wil20}.
This results in the constant $C$ in (\ref{eq: fin dim 1}) not only depending on the dimension of $\tilde \sK_\R$ but also on the representation of $\R$ on $\tilde \sK_\R$.
		
For the moreover part, we compute
\[\begin{aligned}
|\langle \eta\ot g, W(\zeta_1) JW(\zeta_2)^*J (\xi\ot f)\rangle|=|\int_\R  \overline{g(t)} f(t)  \langle JW(\zeta_2)J\eta, W(U_t\zeta_1)\xi\rangle_q  dt|\leq C' q^{ki}\|\zeta_1\|_1\|\zeta_2\|_q,
\end{aligned}\]
	where $C'$ depends on $C$, $f$ and $g$, 
	as $\langle JW(\zeta_2)J\eta, W(U_t\zeta_1)\xi\rangle_q =\langle J\zeta_2, W(\eta)^* W(U_t\zeta_1) W(\xi)\Omega\rangle_q$.
\end{proof}

\begin{lem}\label{lem: wk coarse inte step}
Assuming Notation~\ref{note 1},
	let $\sK_\R\subset \sH_\R$ be a finite dimensional subrepresentation
	and $N:=\Gamma_q(\sK_\R, U)''\subset M$.

Then there exists some $\kappa\in \N$ depending on $U:\R\to \cO(\sK_\R)$ and $q$ such that
	the $c(M)$-$c(M)$ bimodule
	$\oplus_{k\geq \kappa}\cL_k\ot L^2\R$ is a weakly coarse when viewed as an $N$-$N$ bimodule.
\end{lem}
\begin{proof}
Observe that it suffices to show that for any finite dimensional subrepresentation $\sF_\R\subset \sL_\R$,
	any $n, k\in \N$ with $n\geq k\geq \kappa$, any 
	$\xi_0 \in \sF^{\ot n}\cap \cL_k$ and any $f\in C_c(\R)$, 
	the $N$-$N$ bimodule $\cH$ generated by $\xi:=\xi_0\ot f$ is weakly coarse.

Put $\varphi=\chi_{\mid N}$ and 
	one checks that $\xi$ is a left and right $\varphi$-bounded vector as $\sF_\R$ is finite dimensional.
Set $\theta_\xi: N\ni x\mapsto L_\xi^* x L_\xi\in N$,
	which extends to a bounded map  $T_\xi=L_\xi^* R_\xi:L^2(N,\varphi)\to L^2(N,\varphi)$ \cite[Lemma 2]{OOT17}.
Notice that $\langle JyJ\varphi^{1/2} , T_\xi x\varphi^{1/2}\rangle=\langle \xi, x Jy^*J\xi\rangle$
	for $x,y\in N$
	and hence for $\zeta_1\in (L\oplus 0)^{\otimes i}$, $\zeta_2\in (L\oplus 0)^{\otimes j}$,
	we have $\langle J\zeta_2, T_\xi \zeta_1\rangle=\langle \xi, W(\zeta_1)JW(\zeta_2)^*J \xi\rangle$.
It then follows from Lemma~\ref{lem: inequality for wk coarse} and \cite[Lemma 6.5]{Wil20} 
	that $T_\xi\in \B(L^2(N,\varphi))$ is a trace-class operator
	if $\kappa>-\log(\dim(L_\R))/\log(|q|)$
	and hence 
	$$
	N\ot_{\rm min} N^{\op}\ni x\ot y^{\op}\mapsto \langle \xi, x Jy^*J\xi\rangle
	$$
	is continuous,
	which implies the desired result.
Indeed, if $T_\xi=\sum_{i=1}^\infty \lambda_i e_i\ot f_i$ denotes the singular value decomposition of $T_\xi$ with $\lambda_i\in \C$,
	$\{e_i\}$, $\{f_i\}$ orthonormal systems in $L^2(N,\varphi)$, 
	then for any $\sum_{j=1}^d x_j\ot y_j^\op\in N\ot N^{\op}$ one has
\[\begin{aligned}
	&|\sum_{j=1}^d\langle \xi, x_j Jy_j^* J\xi\rangle|
	=|\sum_{j=1}^d \sum_i \lambda_i \langle f_i, x_j\varphi^{1/2}\rangle \langle Jy_j J\varphi^{1/2}, e_i\rangle| \\
	=&|\sum_i \lambda_i\langle f_i\ot \varphi^{1/2}, (\sum_{j=1}^d x_j\ot Jy_j^*J)(\varphi^{1/2}\ot e_i)\rangle| \\
	\leq & \sum_i |\lambda_i|\|\sum_{j=1}^d x_j\ot Jy_j^* J\|=\|T_\xi\|_{S_1}\|\sum_{j=1}^d x_j\ot Jy_j^* J\|.
\end{aligned}\]
\end{proof}

\begin{lem}
Let $(M,\varphi)$ and $(N,\psi)$ be von Neumann algebras with faithful normal states
	and $G\actson (M,\varphi)$ , $H\actson (N,\psi)$ are state-preserving continuous actions of amenable locally compact groups.	
Suppose $\pi: (M\rtimes G)\ot_{\alg} (N\rtimes H)^{\op}\to \B(\cH)$ is a binormal representation.

If $\pi_{\mid M\otimes_{\alg} N^{\op}}$ is $\min$-continuous
	and there exists a weakly dense separable ${\rm C}^*$-subalgebra $A\subset M$ invariant under $G$-action
	such that $A\rtimes_{\rm red}G$ is exact.
Then $\pi$ is $\min$-continuous.
\end{lem}
\begin{proof}
Note that $\pi$ restricts to a representation of $(M\ot_{\min} N^{\op})\rtimes_{\alg} G$ which is continuous with respect to the norm on $(M\ot_{\rm min} N^{\op})\rtimes_{\rm red}G$ by the amenability of $G$.
As $H$ is also amenable, we have $\pi$ extends to  $(M\rtimes_{\rm red}G)\ot_{\min} (N\rtimes_{\rm red} H)^{\op}$.
The result then follows from \cite[Lemma 9.2.9]{BO08} as $\pi$ is binormal and $A\rtimes_{\rm red}G$ is separable and exact.
\end{proof}

\begin{cor}\label{cor: bound for coarse}
With Notation~\ref{note 1},
	let $\sK_\R\subset \sH_\R$ be a finite dimensional subrepresentation
	and $N:=\Gamma_q(\sK_\R, U)''\subset M$.
	
Then for $\kappa\in \N$ satisfying $\kappa>- \log(\dim(K_\R))/\log(|q|)$, one has
	$\oplus_{k\geq \kappa}\cL_k\ot L^2\R$ is a weakly coarse $c_\varphi(N)$-$c_\varphi(N)$ bimodule,
	where $\varphi=\chi_{\mid N}$.
\end{cor}
\begin{proof}
This is a direct consequence of the above two lemmas,  by noting that $\Gamma_q(\sK_\R, U)$ is exact by \cite{KN11}
	as exactness passes to subalgebras,
	which implies $\Gamma_q(\sK_\R, U)\rtimes_{\rm red}\R$ is exact.
\end{proof}

\begin{lem}\label{lem: embedding bimodules}
Let $U:\R\to \cO(\oplus_{i=1}^3 \sH_\R^{(i)})$ be an almost periodic representation with each $\sH_\R^{(i)}$ nontrivial.
Set $M=\Gamma_q(\oplus_{i=1}^3\sH_\R^{(i)}, U)''$ and $\chi$ its $q$-quasi free state.
For each subset $I\subset \{1,2,3\}$, let $M_I=\Gamma_q(\oplus_{i\in I} \sH_\R^{(i)}, U)''\subset M$
	and put $\cM=M\rtimes_{\sigma^\chi}\R$ with semifinite trace $\Tr_\chi$,
		and $\cM_I=M_I\rtimes_{\sigma^\chi}\R\subset \cM$.


Let $p\in L_\chi\R$ be a nonzero finite trace projection and set $N=p\cM p$, $N_I=p\cM_I p$.
Then we may identify $L^2(N_{\hat 3 })\ot_{ N_{1} } L^2(N_{\hat 2})$ as an $N_1$-$N_1$ sub-bimodule of $L^2(N)$,
	where we denote by $\hat i$ the set $I\setminus \{i\}$ and by $i$ the set $\{i\}$.
\end{lem}
\begin{proof}
Define $\theta: N_{\hat 3}\ot_{\rm alg} L^2(N_{\hat 2})\ni a\ot \xi \mapsto a\xi\in L^2(N)$,
	where we view $L^2(N_{\hat 2})$ as a subspace of $L^2(N)$.
It is clear that $\theta$ is $N_1$-bimodular 
	and thus it suffices to show that $\theta$ extends to an embedding on $L^2(N_{\hat 3 })\ot_{ N_{1} } L^2(N_{\hat 2})$.

For $\xi, \eta\in L^2(N_{\hat 2})$ and $a,b\in N_{\hat 3}$, one has
	$\langle \theta(a\ot \xi), \theta(b\ot \eta)\rangle=\langle a\xi, b\eta\rangle=\langle b^*a \xi, \eta\rangle$.
Thus to see $\theta$ extends to an isometry on $L^2(N_{\hat 3 })\ot_{ N_{1} } L^2(N_{\hat 2})$, 
	it is enough to prove that 
	$\langle x \xi, \eta\rangle=\langle E_{N_{1}}(x)\xi, \eta\rangle$ for any $x\in N_{\hat 3}$ and $\xi,\eta\in L^2(N_{\hat 2})$.

We claim that for $x\in C_c(\R, M_{\hat 3})$ and $\xi,\eta\in C_c(\R, \cF_q(\sH^{(1)}\oplus \sH^{(3)}))$, 
	one has
$$
\langle x\xi, \eta\rangle=\langle E_{\cM_1}(x)\xi,\eta\rangle,
$$
where $E_{\cM_1}: \cM\to \cM_1$ is the normal expectation.
Observe that from our claim one also has the same equation holds for $x\in p\cM_{\hat 3} p$
	and $\xi,\eta \in L^2(p\cM_{\hat 2} p)$ by a density argument.
Since ${E_{\cM_1}}_{\mid p\cM_1 p}$ coincides with the normal expectation from $p\cM p$ to $p\cM_1 p$,
	our desired conclusion follows.

To see our claim, we compute
\[
\begin{aligned}
\langle x\xi, \eta\rangle=\int_\R \langle \lambda_sx(s)\xi,\eta\rangle ds
	=\int_\R \Big(\int_\R \langle \sigma_{-t}(x(s))\xi(t),\eta(s+t)\rangle dt \Big)ds.
\end{aligned}
\]
If one has $\langle \sigma_{-t}(x(s))\xi(t),\eta(s+t)\rangle=\langle E_{M_1}\big (\sigma_{-t}(x(s))\big)\xi(t), \eta(s+t)\rangle$,
	then it follows that 
$$
\int_\R \langle \lambda_sx(s)\xi,\eta\rangle ds=\int_\R \langle \lambda_s E_{M_1}(x(s))\xi,\eta\rangle ds=\langle E_{\cM_1}(x)\xi,\eta\rangle.
$$
Therefore, we only need to prove that for $x\in \Gamma_q(\sH_\R^{(1)}\oplus \sH_\R^{(2)}, U)''$, $\xi,\eta\in \cF_q(\sH^{(1)}\oplus \sH^{(3)})$,
	one has $\langle x\xi,\eta\rangle=\langle E_{M_1}(x)\xi,\eta\rangle$, for which one easily verifies by using Wick formulas.
\end{proof}

\begin{proof}[Proof of Proposition~\ref{prop: coarse step lower bound}]
We first set some notations.
For an almost periodic representation $V:\R\to \cO(\mathsf J_\R)$ containing $\sH_\R$ as a subrepresentation,
	we denote by $\cL_k^{\sH\subset \mathsf J}\subset \cF_q(\mathsf J)$ the closed subspace spanned by simple tensors with $k$ legs in $\mathsf J\ominus \sH$
		and the rest in $\sH$.
	
Observe that $\cH:=L^2(p c(\tilde M)p\ominus pc(M) p)=p\big( (\oplus_{n\geq 1}\cL_n^{\sH\subset \sL})\ot L^2\R \big)p$
	by Remark~\ref{rem: bimod identify}. 
Put $\sF_\R= \sL_\R\ominus \sH_\R$ 
	and $\sL^{(2)}_\R=\sL_\R\oplus \sF_\R$, which contains $\sH_\R=\sH_\R\oplus 0\subset \sL_\R\oplus \sF_\R=(\sH_\R\oplus \sF_\R)\oplus \sF_\R$ as a subrepresentation.
It follows that we may apply Lemma~\ref{lem: embedding bimodules} to yield the following embedding of $p c(M) p$-bimodules
\begin{equation}\label{eq: bimod embed}
	L^2(p c(\tilde M) p)\ot_{pc(M)p} L^2(pc(\tilde M)p)\subset L^2(p c(\tilde M_2)p)
\end{equation}
	where $\tilde M_2=\Gamma_q(\sL_\R^{(2)}, U\oplus U)''$,
	and the first $\tilde M$ in (\ref{eq: bimod embed}) is identified with 
	$\Gamma_q((\sH_\R\oplus \sF_\R)\oplus 0)''\subset \tilde M_2$
	while the second $\tilde M$ is viewed as $\Gamma_q((\sH_\R\oplus 0)\oplus \sF_\R)''\subset \tilde M_2$.
Thus we also have $\cH\ot _{p c(M)p}\cH\subset L^2(p c(\tilde M_2) p)$.

Moreover, since $\cH=p\big( (\oplus_{n\geq 1}\cL_n^{\sH\subset \sL})\ot L^2\R \big)p$,
	we actually have
\begin{equation}\label{eq: embed amplify}
	\theta(\cH\ot_{p c(M)p}\cH)\subset p(\oplus_{n\geq 2}\cL_n^{\sH\oplus 0\subset \sL\oplus \sF}\ot L^2\R )p,
\end{equation}
	where $\theta$ is the embedding from Lemma~\ref{lem: embedding bimodules}.
Indeed, for any $x\in C_c(\R, \Gamma_q((\sH_\R\oplus \sF_\R) \oplus 0))$ with $x(s)\Omega\in \oplus_{n\geq 1} \cL_n^{\sH\oplus 0\subset \sL\oplus 0}$
	and 
	$$
	\xi\in  C_c(\R, \oplus_{n\geq 1} \cL_{n}^{(\sH\oplus 0)\oplus 0\subset (\sH\oplus 0)\oplus \sF}),
	$$
	one has
	$x\xi=\int_{\R} \lambda_s x(s)\xi ds$
	while $(x(s)\xi)(t)=\sigma_{-t}(x(s))\xi(t)$,
	which lies in $\oplus_{n\geq 2}\cL_n^{\sH\oplus 0\subset \sL\oplus \sF}$ by checking using Wick formulas.
A density argument then shows (\ref{eq: embed amplify}).

Finally, the conclusion follows by taking tensor power $k$-times, where $k> -\log(\dim(K_\R))/\log(|q|)$
	and applying Corollary~\ref{cor: bound for coarse}.
\end{proof}

Similarly, we have the following variant of Proposition~\ref{prop: coarse step lower bound} without $\R$-actions.
\begin{lem}\label{lem: wk coarse lower bound w/o R action}
Let $q\in (-1,1)\setminus\{0\}$ and $U:\R\to \cO(\sH_\R)$ be an almost periodic representation 
	with a finite dimensional subrepresentation $\sK_\R\subset \sH_\R$.
Set $N=\Gamma_q(\sK_\R, U)''$ and $M=\Gamma_q(\sH_\R, U)''$.

Then there exists some $k\in \N$ depending on $q$ and $\dim(\sK_\R)$ such that
	$L^2(M\ominus N)^{\ot _N^k}$ is weakly coarse as an $N$-$N$ bimodule.
\end{lem}
\begin{proof}
Setting $\sF_\R=\sH_\R\ominus \sK_\R$ and $\cH=L^2(M\ominus N, \chi)$, we use the some notations from the preceding proof.
Since the ideas are mostly identical to ones in the proof of Proposition~\ref{prop: coarse step lower bound}, we only sketch the argument.

We first show that $\cH\ot_N \cH\subset \oplus_{n\geq 2} \cL_n^{\sK\oplus 0\subset \sH\oplus \sF}$ following the same argument of Lemma~\ref{lem: embedding bimodules}.
Indeed, consider $\theta: M\chi^{1/2}\ot L^2M\ni a\chi^{1/2}\ot \xi\mapsto a\xi\in L^2(M_2)$,
	where $M_2=\Gamma_q(\sH_\R\oplus\sF_\R, U\oplus U)''$.
	
We claim that $\theta$ extends to $L^2M\ot_N L^2M$.
Note that $\langle b\eta, a\xi\rangle=\langle \eta, b^*a\xi\rangle=\langle\eta, E_N(b^*a)\xi\rangle$
	for $a,b\in \Gamma_q((\sH_\R\oplus 0)\subset M_2$ and $\xi,\eta\in \cF_q((\sK\oplus 0)\oplus \sF)\subset \cF_q(\sH\oplus \sF)$,
	where $E_N: M\to N$ is the expectation.
As $M\chi^{1/2}\ot L^2M\subset L^2M\ot _N L^2M$ is dense by \cite[IX, Proposition 3.15]{Take03},
	our claim follows.
Taking the restriction, we have $\theta: \cH\ot_N \cH\to \cF_q(\sH\oplus \sL)$.
Moreover, one verifies that $\theta(a\chi^{1/2}\ot \xi)\in \cL_2^{\sK\oplus 0\subset \sH\oplus \sF}$ if $a=W(\eta)$ and $\xi,\eta\in \cL_1^{\sK\subset \sH}$,
	from which we conclude that $\cH\ot_N \cH\subset \oplus_{n\geq 2} \cL_n^{\sK\oplus 0\subset \sH\oplus \sF}$ as $N$-$N$ bimodules.
	
Therefore we have $\cH^{\ot_N^k}\subset \oplus_{n\geq k} \cL_n^{\sK\subset \tilde \sH}$, where $\tilde \sH=\sH\oplus(\oplus_{i=1}^{k-1}\sF)$.
If $k>-\log(\dim(K_\R))/\log(|q|)$, then for any simple tensor $\xi\in \oplus_{n\geq k} \cL_n^{\sK\subset \tilde \sH}$, one has 
	$\overline{N\xi N}$ is weakly coarse by the same proof of Lemma~\ref{lem: wk coarse inte step}.
As such $\xi$ generates  $\oplus_{n\geq k} \cL_n^{\sK\subset \tilde \sH}$ as an $N$-$N$ bimodule, the conclusion follows.
\end{proof}

\section{Structural results for almost periodic $q$-Araki-Woods factors}

Building on the bimodule computation from the previous section, we prove various structural results for almost periodic $q$-Araki-Woods factors 
	in this section.
The foundation of these results is the following dichotomy for subalgebras in continuous cores of $q$-Araki-Woods factors,
	which can be seen as a type ${\rm III}$ version of \cite[Theorem 8.11]{DP23}
	and follows a similar proof.

\begin{thm}\label{thm: dichotomy}
Let $U: \R\to \cO(\sH_\R)$ be an almost periodic representation
	and $q\in (-1,1)\setminus\{0\}$.
Set $M=\Gamma_q(H_\R, U)''$ with its $q$-quasi free state $\chi$ and $p\in L_\chi \R$ a finite trace projection.
Denote by $\X$ the $p c_\chi(M) p$-boundary piece associated with $pL_\chi\R$.

For any (possibly nonunital) von Neumann subalgebra $N\subset p c_\chi(M) p$, we have
	either $N$ is properly proximal relative to $\X^N$,
		or $N$ has an amenable direct summand,
	where $\X^N$ denotes the $N$-boundary piece induced from $\X$.
\end{thm}

We first set some notations and prepare a lemma.

\begin{note}\label{note 2}
Let $U: \R\to \cO(\sH_\R)$ be an almost periodic representation 
	and set $\tilde \sH_\R=\sH_\R\oplus \sH_\R$ and $\tilde U=U\oplus U$.
View $M=\Gamma_q(\sH_\R, U)''\subset \tilde M=\Gamma_q(\tilde \sH_\R, \tilde U)''$ through the inclusion 
	$\sH_\R\oplus 0\subset \sH_\R\oplus\sH_\R$
	and let $\chi$ denote the $q$-quasi free state on $\tilde M$.
Put $(\tilde \cM, \Tr)=(\tilde M\rtimes_{\sigma^\chi}\R, \Tr_{\chi})$, which contains $\cM=M\rtimes_{\sigma^\chi}\R$.
	
For each $s\in [0,1]$ and finite dimensional subrepresentation $\sF_\R\subset \sH_\R$, set
	$A_{s,\sF}=\cos(\pi s/2)P_\sF$, where $P_\sF: \sH_\R\to \sF_\R$ is the orthogonal projection.
Note that as $[P_\sF, U_t]=0$ for all $t\in \R$, one has
$$
V^0_{s,\sF}=\begin{pmatrix}
	A_{s,\sF} & -\sqrt{1-A_{s,\sF}^2} \\
	\sqrt{1-A_{s,\sF}^2} & A_{s,\sF}	
	\end{pmatrix} 
	\in \cO(\sH_\R\oplus\sH_\R)
$$
commutes with $\{\tilde U_t\}_{t\in\R}$ and hence it induces a unitary $V_{s,\sF}:=\cF_q(V^0_{s,\sF})$ on the Fock space $\cF_q(\sH\oplus \sH)$
	that gives rise to $\alpha_{s,\sF}\in \Aut(\tilde\cM, \Tr)$,
	where $\alpha_{s,\sF}=\Ad(V_{s,\sF}\ot \id)\in \B(\cF_q(\tilde \sH)\ot L^2\R)$ \cite[Proposition 1.1]{Hia03}.
Note that ${\alpha_{s,\sF}}$ restrics to identity on $L_\chi\R$ and $JL_\chi\R J$, and $\alpha_{s,\sF}(W(\xi))=W(V_{s,\sF}\xi)$ for $\xi\in \tilde\sH_\C$.
	
\end{note}

Recall the identifications from Remark~\ref{rem: bimod identify}.
\begin{lem}\label{lem: bimodular and compact}
With the above notation, let $p\in L_\chi\R$ be a nonzero finite trace projection
	and consider the u.c.p.\ map $\phi_{s,\sF}: \B(L^2(p\tilde \cM p))\to \B(L^2(p\cM p))$
	given by 
	$$\phi_{s,\sF}(T)= e_{p\cM p} (V_{s,\sF}^*\ot\id)  T (V_{s,\sF}\ot\id)e_{p\cM p},$$
	where $e_{p\cM p}: L^2(p\tilde \cM p)\to L^2(p\cM p)$ is the orthogonal projection.
	
If $\X$ denotes the $p\cM p$-boundary piece associated with $pL_\chi\R$,	
	then $\phi_{s,\sF}(e_{p\cM p})\in \K_{\X}(p\cM p)$ for any $s\in (0,1]$ and $\sF_\R\subset \sH_\R$ finite dimensional.

Moreover, one has
	$\phi_{t, F}$ is almost $p \cM p$ and $Jp \cM pJ$ bimodular in $\|\cdot\|_{\infty,1}$, i.e.,
	$\|\phi_{s,\sF}(xTy)-x\phi_{s,\sF}(T)y\|_{\infty,1}\to 0$ as $s\to 0$ and $\sF_\R\to \sH_\R$
	for any $x,y\in p\cM p\cup Jp\cM pJ$ and $T\in \B(L^2(p \tilde \cM p))$.
\end{lem}
\begin{proof}
For any $0<s\leq 1$ and $\sF_\R\subset \sH_\R$ finite dimensional,
	one computes that
	$e_{\cM} (V_{s,\sF}\ot \id) e_\cM: L^2\cM\to L^2\cM$
	equals $K\ot \id_{L^2\R}$ via the identification $L^2\cM=\cF_q(\sH\oplus 0)\ot L^2\R$,
	where $K=\sum_{n=0}^\infty \cos(\pi t/2)^n P_\sF^n\in \K(\cF_q(\sH\oplus 0))$
	and $P_\sF^n=\cF_q(P_\sF)_{\mid H^{\ot n}_q}: \sH^{\ot n}_q\to \sF^{\ot n}_q$ while $P_\sF^0$ denotes $\id_{\C\Omega}$.
	
Since $V_{s,\sF}\ot \id$ commutes with $\lambda_{\chi}(\R)$ and $\rho_{\chi}(\R)$, we have
	$e_{p\cM p} (V_{s,\sF}\ot \id) e_{p\cM p}$
	coincides with $e_\cM(V_{s,\sF}\ot \id) e_{\cM} p JpJ$ and thus $\phi_{s,\sF}(e_{p\cM p})\in p JpJ \big(\K(\cF_q(\sH\oplus 0))\ot_{\rm alg} \B(L^2\R)\big)JpJp$.
	
Recall that $\X$, the $p\cM p$-boundary piece associated with $pL\chi\R$, is generated by $\{xJyJe_{pL_{\chi}\R}^{p\cM p} \mid x,y\in p\cM p\}$,
	where $e_{pL_\chi\R}^{p\cM p}: L^2(p\cM p)\to L^2(pL_\chi\R)$ is the orthogonal projection,
	which equals to $(P_{\Omega}\ot \id_{L^2\R})pJpJ=P_\Omega\ot \id_{pL^2\R}$ when viewd in $\B(L^2(M,\chi)\ot L^2\R)$,
	where $P_\Omega: \cF_q(\sH\oplus 0)\to \C\Omega$ is the rank-one projection.
It follows that $\phi_{s,\sF}(e_{p\cM p})\in \K_{\X}^{\infty,1}(p\cM p)$
	as we have
	$$
	pJpxpJ (P_{\Omega}\ot \id_{pL^2\R}) JpypJp=pJpJ \big((JxJP_\Omega JyJ)\ot \id_{L^2\R} \big)JpJ p,
	$$
	for any $x,y\in M$.
	
Next one observe that $\alpha_{s,\sF}(x)\to x$ strongly for any $x\in \cM$ as $s\to 0$ and $\sF_\R \to \sH_\R$,
	since $\alpha_{s,\sF}(W(\xi_1\ot \cdots \xi_n))=\cos(\pi s/2)^n W(P_\sF(\xi_1)\ot \cdots \ot P_\sF(\xi_n))$
	for any $\xi_1\ot \cdots \ot \xi_n\in \sH_\C^{\ot n}$
	and $\alpha_{s,\sF}$ is the identity when restricted to $L_\chi\R$.
It then follows that for any $T\in \B(L^2(p\tilde \cM p))$ and $x\in p\cM p$,
	one has $\|\alpha_{s,\sF}(Tx)-\alpha_{s,\sF}(T)x\|_{\infty,1}\to 0$ as
\[\begin{aligned}
\|\phi_{s, \sF}(Tx)-\phi_{s,\sF}(T)x\|_{\infty,1}=\sup_{a,b\in (p\cM p)_1} \langle \widehat{ \alpha_{s,\sF}( b)}, T\big(x \alpha_{s,\sF}(a)-\alpha_{s,\sF}(xa)\big) \hat 1\rangle\leq \|T\| \|x-\alpha_{s,\sF}(x)\|_2.
\end{aligned}\]
The almost left $p\cM p$-modularity and almost $J p\cM p J$-bimodularity follows similarly. 
\end{proof}

\begin{proof}[Proof of Theorem~\ref{thm: dichotomy}]
We follow Notation~\ref{note 2} and denote by $\X$ the $p\cM p$-boundary piece associated with $pL\chi\R$.

Assume $N$ is not properly proximal relative to $\X^N$,
	then one has a nonzero central projection $z\in \cZ(N)$ and an $Nz$-central state
	$\varphi:\tilde \bS_{\X}(p\cM p)\to \C$ such that
	$\varphi_{\mid z \cM z}=\tau_{z\cM z}$ by the discussion in Section~\ref{sec: biexact and pp}.
	
In the following, we will show $Nz$ is amenable
	and thus we may replace $N$ with $Nz\oplus \C(p-z)$
	so that $\varphi:\tilde \bS_\X(p\cM p)\to \C$ is $N$-central and $\varphi_{\mid p\cM p}=\tau_{p\cM p}$

Taking a limit point of $\phi_{s, \sF}$ from the previous lemma yields a u.c.p.\ map $\phi: \B(L^2(p\tilde \cM p))\to (\B(L^2(p\cM p))^\sharp_J)^*$
	that is $p\cM p$ and $J p\cM pJ$-bimodular
	and $\phi(e_{p\cM p})\in (\K_\X(p\cM p)^\sharp_J)^*$ by Lemma~\ref{lem: bimodular and compact}.
Therefore we have that $\phi$ restricts to 
	$$\phi: \B(L^2(p\tilde \cM p\ominus p\cM p))\cap(J p\cM p J)'\to \tilde \bS_\X(p\cM p),$$
	by the proof of \cite[Proposition 9.1]{DKEP22}.

Put $P=p\cM p$ with tracial state $\tau$ and consider the $P$-$P$ bimodule $\cH=L^2(p\tilde \cM p\ominus p\cM p)$, 
	for which we have $_P \cH_P = {_P} \bar{\cH}_P$.
Composing $\phi$ with $\varphi$ yields an $N$ central state on $\B(\cH)\cap (JPJ)'$
	that restricts to a trace on $P$,
	which in turn gives a sequence of unit vectors $\xi_n\in L^2(JPJ'\cap \B(\cH))=\cH\ot _P \bar \cH$
	that is almost bi-tracial for $P$ and almost central for $N$ \cite[Proposition 2.4]{PoVa14a}.

For any $\varepsilon>0$, any nonzero central projection $e\in \cZ(N)$ and any finite collection of unitaries $u_1,\dots u_d\in \cU(Ne)$, 
	we may find a finite dimensional subrepresentation $\sF_\R\subset \sH_\R$
	such that $\|E_\sF(u_i)-u_i\|_2<\varepsilon/2d$, where $E_\sF: p\cM p\to p(\Gamma_q(\sF_\R, U)''\rtimes\R )p=:Q$
	is the conditional expectation. 

By Proposition~\ref{prop: coarse step lower bound}, there exists a constant $\kappa\in \N$ depending on $\dim(\sF_\R)$ and $q$,
	such that $\cH^{\ot_P^k}$ is a weakly coarse $Q$-$Q$ bimodule for all $k\geq \kappa$.
Note that for any $k>\kappa/2$, one has 
	$\xi_n^k:=\xi_n^{\ot_P^k}\in (\cH\ot_P \bar \cH)^{\ot_P^k}=\cH^{\ot_P ^{2k}}$	.
		
It follows that
\[
\begin{aligned}
\|\sum_{i=1}^d u_i\ot u_i^\op\|_{N\ovt N^\op}\geq \|\sum_{i=1}^d E_\sF( u_i)\ot E_\sF(u_i^\op)\|\geq \|\sum_{i=1}^d E_\sF (u_i)e\xi_n^{k} E_\sF(u_i)^*\|/\|e\xi_n^k\|.
\end{aligned}
\]

Since $\xi_n^{k}$ is almost central for $\{u_1,\dots, u_d\}$ as well as almost bi-tracial for $P$, 
	we have $\|[E_\sF(u_i)^*,\xi_n^{k}]\|<\varepsilon/d$ for $n$ large enough.
Thus we have
$$
\|\sum_{i=1}^d E_\sF (u_i)e\xi_n^{k} E_\sF(u_i)^*\|\geq \|\sum_{i=1}^d E_\sF(u_i) e E_\sF(u_i)^*\xi_n^{k}\|-\varepsilon=\|\sum_{i=1}^d E_\sF(u_i) e E_\sF(u_i)^*\|_2-2\varepsilon\geq d\tau(e)-3\varepsilon.
$$
Since $\|e\xi_n^k\|\to \tau(e)$ as $n\to \infty$ and $\varepsilon>0$ is arbitrary, we have $\|\sum_{i=1}^d u_i\ot u_i^\op\|_{N\ovt N^\op}=d$.
As $e\in \cZ(N)$ is also arbitrary, we conclude that $N$ is amenable by \cite{Ha85}.
\end{proof}

%

We derive from Theorem~\ref{thm: dichotomy} the solidity of $q$-Araki-Woods factors in the almost periodic case.
In the special case that the representation is finite dimensional,	
	the following gives a new proof for solidity
	without relying on the main result of \cite{Ku22} as in \cite{KSW23}.

\begin{thm}\label{thm: solid}
For any almost periodic representation  $U: \R\to \cO(\sH_\R)$ and $q\in (-1,1)$,	
	the $q$-Araki-Woods factor $M=\Gamma_q(\sH_\R, U)''$ is solid, i.e., 
	for any diffuse von Neumann subalgebra $A\subset M$ with expectation,
	one has $A'\cap M$ is amenable.
\end{thm}
\begin{proof}
Note that we may assume $q\neq 0$ and $M$ is of type ${\rm III}_1$ by amplifying $(\sH_\R, U)$ since solidity passes to subalgebras with expectation.

Let $Z\subset A$ be a diffuse abelian von Neumann subalgebra with expectation by \cite[Theorem 11.1]{HaSt90}.
Put $E_A^M: M\to A$ and $E_Z^A: A\to Z$ to be the expectations, and $\varphi_0$ a faithful normal state on $Z$.
Thus $\varphi=\varphi_0\circ E_Z^A\circ E_A^M$ is a faithful normal state on $M$
	such that $\sigma^\varphi$ globally preserves $A$ and $Z$, and hence also $B:=(A'\cap M)\vee Z$.

To see $A'\cap M$ is amenable, it suffices to show $pc_\varphi(B)p$ is amenable for any nonzero finite trace projection $p\in L_\varphi\R$.
To this end, first note that $Z$ commutes with $L_\varphi\R$ as $Z$ is abelian and hence $[Z, c_\varphi(B)]=0$.

Take $p\in L_\varphi\R$ an arbitrary nonzero finite trace projection and a sequence of unitaries $\{u_n'\}\subset Z$ that goes to $0$ weakly.

Since $c_\chi(M)$ is a type ${\rm II}_\infty$ factor, we may find a finite trace projection $q\in L_\chi\R$ and 
	a unitary $v\in c_\chi(M)$ such that 
	$\Pi_{\chi,\varphi}(p) =v^*qv$.	
It follows that for any $x\in v \Pi_{\chi,\varphi}(p c_\varphi(B) p) v^*=q v \Pi_{\chi,\varphi}(c_\varphi(B))v^* q=:B_0$,
	one has $[u_n, x]= 0$, where $u_n=v u_n' v^*$.

Set $M_0=q c_\chi(M) q$ and 	
	consider a u.c.p.\ map $\Phi: \B(L^2(M_0))\to \B(L^2(M_0))$ given a point-weak$^*$ limit point of $\{\Ad(u_nq)\}_{n\in \N}$.
It is clear that $\Phi(x)=x$ for any $x\in JM_0J$ and $x\in B_0$
	and $\Phi$ is continuous in $\|\cdot\|_{\infty,1}$.

Moreover, \cite[Proposition 5.3]{HoRa15} shows that for any $x,y\in M_0$ one has
$$
\|E_{qL_\chi\R}(xqu_n qy)\|_2\to 0,
$$
which implies that $\Phi(K)=0$ for any $K\in \K_\X (M_0)$ by the proof of \cite[Lemma 6.12]{DP23},
	where $\X$ denotes the $M_0$-boundary piece associated with $qL_\chi\R$.

Denote by $e: L^2(M_0)\to L^2(B_0)$ the orthogonal projection given by the conditional expectation $E: M_0\to B_0$.
Recall that the induced $B_0$-boundary piece $\Y:=\X^{B_0}$	is  
	$\overline{e(\K_\X(M_0))e}\subset \B(L^2(B_0))$.
As $\{u_nq\}\subset B_0$, one has $\Ad(e)$ commutes with $\Phi$ and hence
$$
\Psi:=\Phi\circ \Ad(e): \B(L^2(B_0))\to \B(L^2(B_0))
$$
satisfies that $\Psi(K)=0$ for any $K\in \K_\Y^{\infty,1}(B_0)$ and $\Psi(x)=x$ for any $x\in B_0\cup JB_0J$.

Therefore, the u.c.p.\ map $\Psi$ restricts to a $B_0$-bimodular map
$$
\Psi: \bS_\Y(B_0)\to B_0.
$$

If $B_0$ were not amenable, then one would have a nonzero central projection $z\in B_0$ such that $zB_0$ has no amenable direct summand.
However, 
$$\tau_z\circ \Psi: z\bS_\Y(B_0)z=\bS_{z\Y z}(zB_0)\to \C$$ 
	is a $zB_0$-central state that restricts to $\tau_z$ on $zB_0$,
	where $\tau_z=\tau(z\cdot)$ and $\tau$ a trace on $B_0$,
	i.e., $zB_0$ is not properly proximal relative to $z\Y z$.
Note that the $B_0z$-boundary piece $z\Y z$ is $\X^{zB_0}$, the $zB_0$-boundary piece induced from $\X$,
	and thus by Theorem~\ref{thm: dichotomy} $zB_0$ has an amenable direct summand, which is a contradiction.
\end{proof}

It was shown in \cite{KSW23} that $\Gamma_q(\sH_\R, U)''$ is full if $2\leq \dim(\sH_\R)<\infty$
	while the fullness in the case that $U:\R\to \cO(\sH_\R)$ is an infinite dimensional almost periodic representation
	was treated in \cite{KW24}.
(See also \cite{HoIs20} for fullness in the weakly mixing case.)
We obtain the following generalization. 
	
\begin{cor}\label{cor: fullness}
For any $U: \R\to \cO(\sH_\R)$ almost periodic representation 
	and $q\in (-1,1)$,
	any nonamenable subfactor $N\subset \Gamma_q(\sH_\R, U)''$ with expectation is full.
\end{cor}
\begin{proof}
The argument is almost identical to the above one so we only sketch the proof.
Note that we may assume $M=\Gamma_q(\sH_\R, U)''$ is of type ${\rm III}_1$. Let $\varphi$ be a faithful normal state on $M$ and 
	$N\subset M$ be a nonamenable subfactor with $\varphi$-preserving expectation.
By \cite[Corollary 2.6]{HoRa15}, there exists a sequence of unitaries $u_n\in \cU(N)$ such that $u_n\to 0$ weakly,
	$\|[u_n,\varphi]\|\to 0$ and $[u_n,x]\to 0$ strongly for any $x\in N$.
One may then proceed exactly as in the above proof (as one only needs asymptotic commutation for $u_n$)
	and conclude $N$ is amenable.
\end{proof}

We also obtain the strong solidity of almost periodic $q$-Araki-Woods factors.
To this end, we follow the exact same argument of \cite{BHV18} and apply Theorem~\ref{thm: dichotomy} at the end to conclude the proof.

\begin{thm}\label{thm: strong solid}
For any almost periodic representation $U: \R\to \cO(\sH_\R)$  and $q\in (-1,1)$,
	the $q$-Araki-Woods factor $\Gamma_q(\sH_\R, U)''$ is strongly solid.
\end{thm}

\begin{proof}
We may assume $q\neq 0$ and $M=\Gamma_q(\sH_\R, U)''$ is of type ${\rm III}_1$ 
	as strongly solidity passes to von Neumann subalgebras with expectations.

Let $Q\subset M$ be a diffuse amenable von Neumann subalgebra with expectation $E: M\to Q$ and $\psi$ a faithful normal state on $M$ with $\psi\circ E=\psi$.
Set $P:=\cN_M(Q)''$ and assume by contradiction that $P$ is nonamenable.

Arguments in \cite[Proof of the main theorem]{BHV18} shows that we may assume $Q'\cap M=\cZ(Q)$ as $M$ is solid by Theorem~\ref{thm: solid},
	 $\cN_{c_\psi(M)}(c_\psi(Q))''$ has no amenable direct summand
	 and $pQ_0p\not\prec_{M_0} L_\chi\R$ for some finite trace projection $p\in L_\chi\R$,
	 where $M_0=c_{\chi}(M)$, $Q_0=\Pi_{\chi,\psi}(c_\psi(Q))$, $P_0=\Pi_{\chi,\psi}(\cN_{c_\psi(M)}(c_\psi(Q))'')=\cN_{M_0}(Q_0)''$.

Let $M_1= p M_0 p$, $Q_1=p Q_0 p$ and $P_1=s\cN_{M_1}(Q_1)''$, the stable normalizer of $Q_1$ in $M_1$ in the sense of
	\cite[Definition 3.1]{BHV18}, which contains $p P_0 p$.
The following lemma together with \cite[Proposition 5.3]{HoRa15}, \cite{ABW18} and \cite{AD95} implies that $P_1$ is not properly proximal relative to $\X^{P_1}$,
	where $\X$ is the $M_1$-boundary piece associated with $pL_\chi\R$,
	and thus by Theorem~\ref{thm: dichotomy} that $P_1$, and hence $P_0$, must have an amenable direct summand,
	which is a contradiction.
\end{proof}

The following is a consequence of \cite[Proposition 3.6]{BHV18} and generalizes \cite[Theorem 6.11]{DKEP22}.

\begin{lem}
Let $(M,\tau)$ be a tracial von Neumann algebra with CMAP, $A\subset M$ an amenable von Neumann subalgebra,
	$Q\subset M$ a von Neumann subalgebra. Denote by $\X$ the $M$-boundary piece associated with the subalgebra $Q$.
	
If $A\not\prec_{M} Q$, then the von Neumann subalgebra $P$ generated by  $s\cN_M(A)=\{x\in M\mid xAx^*\subset A,\ x^* A x\subset A\}$
	is not properly proximal relative to the induced $P$-boundary piece $\X^P$.
\end{lem}
\begin{proof}
Set $\varphi:\B(L^2M)\to \C$ to a the weak$^*$ limit of $\B(L^2M)\ni T\mapsto \langle \xi_n, (T\ot 1)\xi_n\rangle$,
	where $\{\xi_n\}\in L^2(A\ovt A^\op)$ is a net of  positive vectors given by \cite[Proposition 3.6]{BHV18}.
It is clear that $\varphi$ is $A$-central and restricts to the canonical traces on $M$ and $JMJ$.
Moreover, since $A\not\prec_{M} Q$, we may find a sequence of unitaries $u_n\in \cU(A)$ such that for any $K\in \K_\X^{\infty,1}(L^2M)$,
	we have $\|1/N\sum_{n=1}^N u_n^* Ku_n\|_{\infty,1}\to 0$ as $N\to\infty$ by the proof of \cite[Theorem 6.11]{DKEP22}.
It then follows that $\varphi_{\mid \K^{\infty,1}_\X(M)}=0$.
Viewing $\B(L^2P)=e_P \B(L^2M) e_P\subset \B(L^2M)$, we further have $\varphi$ vanishes on $\K_{\X^P}^{\infty,1}(P)$
	as $(e_P\ot \id)\xi_n=\xi_n$ and $\X^P=\overline{e_P(\K_\X(M))e_P}$.

We claim that $\varphi(vT)=\varphi(Tv)$ for any $T\in \bS_{\X^P}(P)\subset e_P\B(L^2M)e_P$ and partial isometry $v\in s\cN_M^0(A)$.
Indeed, for any $\varepsilon>0$ there exists $S(v,\varepsilon)\in M\odot M^\op$
	such that $\limsup_{n\to \infty}\|\delta_{n,\varepsilon}\|<\varepsilon$
		and  $\limsup_n \|\bar\delta_{n,\varepsilon}\|<\varepsilon$
	by \cite[Proposition 3.6]{BHV18},
	where  $\delta_{n,\varepsilon}=(v\ot 1)\xi_n-(1\ot v^\op) JS(v,\varepsilon)J\xi_n $
	and $\bar\delta_{n,\varepsilon}=(v^*\ot 1)\xi_n-(1\ot \bar v) JS(v,\varepsilon)^*J \xi_n$.

We set $e: L^2M\ot L^2M^\op \to L^2P\ot L^2P^{\op}$ and $E: M\ovt M^{\op}\to P\ovt P^\op$ compute
\[\begin{aligned} 
&\varphi(Tv)=\lim_{n} \langle \xi_n, (Tv\ot 1)\xi_n\rangle
	=\lim_{n}\Big(\langle \xi_n, (T\ot 1)(1\ot v^\op) JE(S(v,\varepsilon))J \xi_n\rangle+ \langle \xi_n, (T\ot 1)\delta_{n,\varepsilon}\rangle\Big)\\
=& \lim_{n} \Big( \langle (1\ot \bar v)JE(S(v,\varepsilon)^*)J \xi_n,(T\ot 1)\xi_n\rangle+\langle (1\ot \bar v)\xi_n, [T\ot 1, JE(S(v,\varepsilon))J]\xi_n\rangle
	+ \langle \xi_n, (T\ot 1)\delta_{n,\varepsilon}\rangle)\Big)\\
=& \lim_{n} \Big(\langle\xi_n, (vT\ot 1)\xi_n\rangle+ \langle e(\bar\delta_{n,\varepsilon}), (T\ot 1)\xi_n\rangle+
	\langle (1\ot \bar v)\xi_n, [T\ot 1, JE(S(v,\varepsilon))J]\xi_n\rangle
	+ \langle \xi_n, (T\ot 1)\delta_{n,\varepsilon}\rangle)\Big)\\
=& \varphi(vT) +\lim_n \Big( \langle e(\bar\delta_{n,\varepsilon}), (T\ot 1)\xi_n\rangle+ \langle \xi_n, (1\ot v^\op) [T\ot 1, JE(S(v,\varepsilon))J]\xi_n\rangle
	+ \langle \xi_n, (T\ot 1)\delta_{n,\varepsilon}\rangle)\Big).
\end{aligned}\]
Note that $(1\ot v^\op)[T\ot 1, JE(S(v,\varepsilon))J]=\sum_{i=1}^d K_i\ot S_i$
	where $S_i\in {\rm C}^*(P, P^\op)$ and $K_i\in \K^{\infty,1}_{\X^P}(P)$,
	since $T\in \bS_{\X^P}(P)$ and $JE(S(v,\varepsilon))J\in J(P\odot P^\op )J$.
As $\lim_{n}\langle \xi_n, (K\ot 1)\xi_n\rangle=\varphi(K)=0$
	for any $K\in \K_{\X^P}(P)_+$,
	we have $\lim_n\langle\xi_n, S\xi_n\rangle=0$ for any $S\in \K_{\X^P}(P)\otimes_{\rm min} \B(L^2P^\op)$.
Using the fact that $\|a\xi_n\|\to \|a\|_2$ for $a\in M$ or $JMJ$,
	we conclude that $\lim_n\langle \xi_n, (1\ot v^\op) [T\ot 1, JE(S(v,\varepsilon))J]\xi_n\rangle=0$.
The claim then follows as $\varepsilon>0$ is arbitrary.

Lastly, as partial isometries in $s\cN_M^0 (A)$ generates $P$ and $\varphi_{\mid P}=\tau$,
	we conclude that $\varphi:\bS_{\X^N}(P)\to \C$ is a $P$-central state with $\varphi_{\mid P}=\tau$.
\end{proof}

\section{Non-isomorphism results}\label{sec: iso}

\subsection{Non-biexactness for $q$-Araki-Woods factors}

In this section we exploit the idea of \cite{BCKW23} and demonstrate non-biexactness for $q$-Araki-Woods factors 
	whose associated representations satisfy certain infinite dimensional condition,
	which in turn leads to non-isomorphism results with free Araki-Woods factors as well as $q$-Araki-Woods factors
	with finitely many generators.
These results are inspired by \cite{Cas23} although our approach is different, 
	as in \cite{Cas23} the notion of ${\rm W}^*$AO was used to distinguish $q$-Gaussian with infinite generators from free group factors,
	while in the current non-tracial setting the notion of biexactness is used.
See \cite[Section 7.3]{DP23} for the connections between these two notions.

The following is based on norm estimates of elements in the min-tensor products of $q$-Araki-Woods algebras from \cite{Nou04, Hia03}.
\begin{lem}\label{lem: non-AO lower bound}
Let $q\in (-1,1)\setminus\{0\}$ and $U:\R\to \cO(\sH_\R)$ a representation.
For any almost periodic subrepresentation $\sK_\R\subset \sH_\R$, denote by $\lambda_{\sK}$ the supremum of eigenvalues of the generator of 
	$U:\R\to \cO(\sK_\R)$.
Set $M=\Gamma_q(\sH_\R, U)''$ with $\chi$ its $q$-quasi free state.
Then the following are true.
\begin{enumerate}
	\item If there exists a finite dimensional almost periodic subrepresentation $\sK_\R$ of $\sH_\R$ such that $2 q^4 \dim(\sK_\R)>\lambda_\sK$,
	then there exist a vector $\xi\in \sH^{\ot 2}$ invariant under $\{\cF_q(U_t)\}_{t\in\R}$ with $\|\xi\|_U=1$, 
		a finite collection of elements $\{x_i\}_{i=1}^d\subset M$
			and $\delta>0$
			such that $|\langle \xi, \sum_{i=1}^d x_i Jx_iJ \xi\rangle|>(1+\delta)\|\sum_{i=1}^d x_i \ot (x_i^*)^{\op}\|$. \label{item: ap 1}
	\item  If there exists a finite dimensional almost periodic subrepresentation $\sK_\R$ of $\sH_\R$ such that $q^8 \dim(\sK_\R)>\lambda_\sK^4$,
		then we may assume $\{x_i\}_{i=1}^d$ from (\ref{item: ap 1}) lies in the centralizer $M^\chi$.\label{item: ap 2}
	\item If the weakly mixing part of $\sH_\R$ is nontrivial, then the same conclusion of (1) holds.\label{item: wm}
\end{enumerate}
\end{lem}
\begin{proof}
(\ref{item: ap 1})
Set $(\sH_\R^{ap}, U_{\mid H_\R^{ap}})=\oplus_{n=1}^N (\sH_\R^n, U_n)$ to be the almost periodic part of $(\sH_\R, U)$,
	where $N\in \N\cup \{\infty\}$, 
	$\sH_\R^n=\R^2$ and the eigenvalues of the generator of $U_n$ are $\lambda_n$ and $\lambda_n^{-1}$, where $\lambda_n\geq 1$.
Let $e_1^n, e_2^n\in \sH_n=\sH_\R^n+i \sH_\R^n$ be unit eigenvectors corresponding to $\lambda_n$ and $\lambda_n^{-1}$, respectively.
For any $d<N$ and $k\in \N$, set
$$
E_{k,d}=\{\xi/\|\xi\|_U \mid \xi=\xi_1\ot \cdots \xi_k,\ {\rm where\ each\ } \xi_i\in \cup_{n=1}^d\{e_1^n, e_2^n\}\},
$$
i.e., $E_{k,d}$ is the set of unit eigenvectors of $\cF_q(A)$ restricted on $(\oplus_{n=1}^d \sH_n)^{\ot k}$, where $A$ is the generator of $U$.

Observe that for any $m>d$ and $\xi\in E_{k,d}$, one has
\[\begin{aligned}
&\langle e_1^m\ot e_2^m, W(\xi) JW(\xi)J e_1^m\ot e_2^m\rangle=\langle JW(S\xi)Je_1^m\ot e_2^m, \xi\ot e_1^m\ot e_2^m\rangle\\
	=&\langle e_1^m\ot e_2^m\ot A^{-1/2}\xi, \xi\ot e_1^m\ot e_2^m\rangle
	=q^{2k}\langle A^{-1/2}\xi, \xi\rangle.
\end{aligned}\]

It follows that
$$
\langle e_1^m\ot e_2^m, \sum_{\xi\in E_{k,d}} W(\xi) JW(\xi)J e_1^m \ot e_2^m\rangle=q^{2k} (\sum_{n=1}^d (\lambda_n^{1/2}+\lambda_n^{-1/2}))^k\geq (2dq^2)^k.
$$
On the other hand, by \cite{Nou04} we have
$$\|\sum_{\xi\in E_{k,d}} W(\xi)\ot JW(\xi)J\|\leq C_{|q|}(k+1)^2 (2d)^{k/2} T^k,$$
where $T=\max\{\lambda_1,\cdots, \lambda_d\}$.
Therefore, if $2q^4d/T^2>1$, then we may find $\delta>0$ and sufficiently large $k\in \N$
	such that the desired conclusion holds, with $\{x_i\}=\{W(\xi)\mid \xi\in E_{k,d}\}$
	and $\xi=e_1^m$ for any $m>d$.
	
(\ref{item: ap 2}) The argument is exactly the same as in (\ref{item: ap 1}), by replacing the set $E_{k,d}$ with the subset $F_{k,d}\subset E_{k,d}$ that consists of eigenvectors with eigenvalue $1$ (if $k$ is even).
Notice that $|F_{k,d}|>d^{k/2}$.
Thus the same proof of (1) shows that for any $m>d$, one has
$$
\langle e_1^m\ot e_2^m, \sum_{\xi\in F_{k,d}} W(\xi) JW(\xi)J (e_1^m\ot e_2^m)\rangle=q^{2k} |F_{k,d}|,
$$
and
$$
\|\sum_{\xi\in F_{k,d}} W(\xi)\ot JW(\xi)J\|\leq C_{|q|} (k+1)^2 T^k |F_{k,d}|^{1/2}.
$$
It follows that if $q^2 d^{1/4}/T>1$, then the conclusion holds.

(\ref{item: wm}) We may assume $U: \R\to \cO(\sH_\R)$ is weakly mixing. 
Denote by $A$ the generator of $U$ and we may find $1<a<b<\infty$ such that $[a,b]\subset {\rm sp}(A)$.
Take $0<\delta<\min\{b-a, 1\}$ and set
$$K_n=E_A\big([a+\frac{\delta}{n+1}, a+\frac{\delta}{n})\big)H_\C,\ \ 
	K_{-n}=E_A \big( ( ( a+\frac{\delta}{n})^{-1}, (a+\frac{\delta}{n+1})^{-1}]\big)H_\C.$$

For $m>n$, take vectors $e_m\in K_m$, $e_n\in K_n$ and compute
$$
\langle e_m, W(e_n) JW(e_n)J e_m\rangle=\langle e_m\ot A^{-1/2} e_n, e_n\ot e_m\rangle=q\|e_m\|^2\langle \Omega, W(e_n) JW(e_n) J\Omega\rangle,
$$
as $\langle e_n, e_m\rangle=\langle A^{-1/2} e_n, e_m\rangle=0$.
Similarly, for any $e_{-n}\in K_{-n}$ one has
$$\langle e_m, W(e_{-n}) JW(e_{-n})J e_m\rangle=q\|e_m\|^2\langle \Omega, W(e_{-n})JW(e_{-n})J\Omega\rangle.$$

Therefore, by the proof of \cite[Theorem 2.3]{Hia03}, we may find $e_n\in K_n$, $e_{-n}\in K_{-n}$
	for $n=1,\dots, N$ such that
$$
\langle\Omega, \sum_{n=1}^N (W(e_n) JW(e_n)J+W(e_{-n})JW(e_{-n})J)\Omega\rangle\geq \frac{a^{1/2}(1+(a/b)^{1/2})N}{a+1}=: r(N),
$$
while
$$
\|\sum_{n=1}^N( W(e_n) \ot JW(e_n)J+ W(e_{-n})\ot JW(e_{-n})J)\|\leq \frac{4}{1-|q|}(\frac{b+1}{a+1} N)^{1/2}=: s(N).
$$
It follows that we may find $N\in \N$ such that $|q| r(N)>(1+\delta)s(N)$ for some $\delta>0$
	and hence the desired conclusion holds by taking $\{x_i\}=\{W(e_n), W(e_{-n})\}_{n=1}^N$
		and $\xi$ to be any unit vector in $K_m$ with $m>N$.
\end{proof}

We now give sufficient conditions for $q$-Araki-Woods factors to be non-biexact.
Here we denote by $\sH_\R^{ap}$ and $\sH_\R^{wm}$ the almost periodic and weakly mixing part of a representation $U: \R\to \cO(\sH_\R)$,
	respectively. 
Note that whenever an almost periodic representation of $\R$ is infinite dimensional and has bounded spectrum, 
	the following two conditions are satisfied.

\begin{thm}\label{thm: non biexact}
Let $q\in (-1,1)\setminus\{0\}$ and $U:\R\to \cO(\sH_\R)$ an infinite dimensional representation.
Set $M=\Gamma_q(H_\R, U)''$ with $\chi$ its $q$-quasi free state and $A$ the generator of $U$.

If there exists $T>0$ such that $2q^2 \dim( E_A([1,T]) \sH_\C^{ap})>T$ or $\sH_\R^{wm}$ is nontrivial, then $M$ is not biexact.

Moreover, if there exists $T>0$ such that $q^8 \dim(E_A([1,T]) \sH_\C^{ap})>T^4$ and $\dim(\sH_\R^{ap})=\infty$, then $M^\chi$ is not biexact.
\end{thm}
\begin{proof}
First we assume $\sH_\R^{wm}$ is nontrivial.
Observe that from  (\ref{item: wm}) of Lemma~\ref{lem: non-AO lower bound} one has $\{x_i\}_{i=1}^d\subset M$ and $\delta>0$ such that
$$
|\langle e_m, \sum_{i=1}^d x_i Jx_i J e_m\rangle|>(1+\delta)\|\sum_{i=1}^d x_i\ot Jx_i J\|,
$$
for any $m$ large enough and $e_m\in K_m$, where the notation $K_m$ is from Lemma~\ref{lem: non-AO lower bound}, (\ref{item: wm}).
It follows that there exists a sequence $a_m:=W(e_m)\in M$
	such that 
$$
M\ot M^{\op}\ni a\ot b^{\op}\mapsto \lim_{m\to \cU} \langle W(e_m)\Omega, a Jb^*J W(e_m)\Omega\rangle
$$
is not min-continuous for any $\cU\in \beta\N\setminus \N$.

Furthermore, note that for $\cU\in \beta\N\setminus \N$, the element $a:=(a_m)_m\in \ell^\infty(\N, M)$ lies inside $\cM^\cU(M)$
	as $\{e_m\}\subset E_A( (b^{-1}, b))H_\C$ and hence $\sigma_{-i/2}^\chi(a_m)=W(A^{-1/2}e_m)$ is uniformly bounded.
Moreover, since $\{e_m\}$ is an orthonormal set, one has $W(e_m)\to 0$ weakly and hence $a\in M^\cU$ with $E_M(a)=0$,
	where $E_M: M^\cU\to M$ is the normal expectation.
We thus conclude that the $M$-$M$ bimodule $L^2(M^\cU\ominus M, \chi^\cU)$ is not weakly contained in the coarse bimodule,
	which implies that $M$ is not biexact.
	
We now consider the case that there exists $T>0$ such that $2q^2 \dim( E_A([1,T]) \sH_\C^{ap})>T$.
Note that from the proof of Lemma~\ref{lem: non-AO lower bound}, (\ref{item: ap 1}), we have $\{x_i\}_{i=1}^d\subset M$ and $\delta>0$ such that
	for any unit vector $\xi \in (\sH_\C\ominus \sK_\C)^{\ot 2}$, where $\sK_\R \subset \sH_\R^{ap}$ is a certain finite dimensional subrepresentation,
	one has
$$|\langle \xi, \sum_{i=1}^d x_i Jx_i J \xi \rangle|>(1+\delta)\|\sum_{i=1}^d x_i\ot Jx_i J\|.$$
Since $\sH_\R$ is infinite dimensional, we may assume $\dim(\sH_\R^{ap})=\infty$.
One may then take a sequence $\xi_m\in \sH_\C^{\ot 2}$ to be invariant under $\{\cF_q(U_t)\}_{t\in\R}$
	(as in the proof of Lemma~\ref{lem: non-AO lower bound}, (\ref{item: ap 1})) that goes to $0$ weakly.
Then it follows that $a:=(W(\xi_m))^\cU\in (M^\chi)^\cU\subset M^\cU$
	and again we have $L^2(M^\cU\ominus M)\not\prec L^2M\ot L^2M$.
	
Finally, the moreover part follows the exact same argument, by noticing that Lemma~\ref{lem: non-AO lower bound}, (\ref{item: ap 2})
	provides us elements in $M^\chi$ 
	and thus we have $L^2(N^\cU\ominus N)\not\prec L^2N\ot L^2N$,
	where $N=M^\chi$.
\end{proof}

As a consequence, combing this theorem with Lemma~\ref{lem: biexact prelim}, we partly resolve \cite[Conjecture 2.11]{KSW23}.

\begin{cor}
Let $U:\R\to \cO(\sH_\R)$ be an infinite dimensional representation such that its weakly mixing part is nontrivial
	or its almost periodic part is infinite dimensional and has bounded spectrum
	and $q\in (-1,1)\setminus \{0\}$.	
\begin{enumerate}
	\item The $q$-Araki-Woods factor $\Gamma_q(\sH_\R, U)''$ is not isomorphic to any free Araki-Woods factors.
	\item For any $q'\in (-1,1)$ and finite dimensional representation $V: \R\to \cO(\sK_\R)$,
		one has $\Gamma_q(\sH_\R, U)''$ is not isomorphic to $\Gamma_{q'}(\sK_\R, V)''$.
\end{enumerate}
\end{cor}

In particular, almost periodic $q$-Araki-Woods factors can not be classified by Connes' Sd invariant 
	for $q\in (-1,1)\setminus\{0\}$, contrasting to the free Araki-Woods case \cite{Sh97}.


\subsection{Non-isomorphic inclusions}
As we have seen from the previous section that the class of almost periodic $q$-Araki-Woods factors is different from
	the class of almost periodic free Araki-Woods factors,
	by considering infinite dimensional representations of $\R$.
In this section, we observes that these two classes are still different in a certain sense even if we only consider 
	free and $q$-Araki-Woods factors associated with finite dimensional representations.
		
In fact, it follows easily from \cite{Sh97} that for a given ergodic finite dimensional representation $U: \R\to \cO(\sH_\R)$
	and a subrepresentation $\sK_\R\subset \sH_\R$,
	the natural inclusion $N:=\Gamma(\sK_\R, U)''\subset M:=\Gamma(\sH_\R, U)''$ is completely classified by
	the (non-closed) subgroups generated by the spectrum of the generator of $\{{U_t}_{\mid K_\R}\}_{t\in \R}$ and $\{U_t\}_{t\in \R}$, respectively.
In contrast, we have the following theorem for $q$-Araki-Woods factors.
In the following, by inclusions of $q$-Araki-Woods factors we always mean the inclusions of factors
	arising from inclusions of representations.

\begin{thm}\label{thm: non-iso inclusions}
Let $U:\R\to \cO(\sH_\R)$ be a finite dimensional representation and $q\in (-1,1)\setminus \{0\}$.
Then there exist a finite dimensional representation $V: \R\to \cO(\sK_\R)$ containing $U: \R\to \cO(\sH_\R)$
	and a subrepresentation $\sL_\R\subset \sK_\R$ such that the inclusion
	$(\Gamma_q(\sH_\R, U)''\subset \Gamma_q(\sK_\R, V)'')$ is not isomorphic to 
	$(\Gamma_q(\sL_\R, V_{\mid L_\R})''\subset \Gamma_q(\sK_\R, V)'')$.
\end{thm}

In particular, if we fix $U:\R\to \cO(\R^2)$ a representation
	and denote by $M_n^q=\Gamma_q(\oplus_{i=1}^n (\R^2, U))''$ for  $n\in \N$,
	then the preceding theorem shows that for any $d\in \N$, there exist $m,n\in \N$
	with $d,m<n$ such that $M_d^q\subset M_n^q$ is not isomorphic to $M_m^q\subset M_n^q$
	if $q\in (-1,1)\setminus\{0\}$.
However, if $U$ is ergodic, one always has $M_d^0\subset M_n^0$ and $M_m^0\subset M_n^0$ are isomorphic for any $d,m<n$ by \cite{Sh97}.

We need the following consequence of \cite{Con74}.
Its proof is similar to the argument in \cite[Theorem F]{HSV19}.

\begin{lem}\label{lem: almost periodic state}
Let $M$ be a factor with faithful normal states $\varphi$, $\psi$ and $N\subset M$ a subfactor with $E_N^\varphi: M\to N$
	(resp.\ $E_N^\psi: M\to N$) $\varphi$-preserving (resp.\ $\psi$-preserving) expectation.	
	
Suppose that $N$ is a full type ${\rm III}$ factor with $\varphi$, $\psi$ almost periodic on $N$
	such that $N^\varphi$ and $N^\psi$ are factors.
Then there exist nonzero projections $p\in N^\varphi$ and $q\in N^\psi$
	such that $(pNp\subset pMp, \varphi_p)$ and $(qNq\subset qMq, \psi_q)$ are isomorphic,
	where $\varphi_p=\varphi(p\cdot p)/\varphi(p)$ and $\psi_q$ defined similarly. 
\end{lem}
\begin{proof}
Set $\cM=\B(\ell^2\N)\ovt M$, $\cN=\B(\ell^2\N)\ovt N$, $\tilde\varphi=\Tr\ot \varphi$ and $\tilde\psi=\Tr\ot \psi$,
	where $\Tr$ denotes the canonical trace on $\B(\ell^2\N)$.
	
By assumption, the weights $\tilde\varphi_{\mid \cN}$ and $\tilde \psi_{\mid \cN}$ are ${\rm Sd}(\cN)$-almost periodic weight \cite[Lemma 4.8]{Con74}
	and hence there exist $u\in \cU(\cN)$ and $\alpha>0$ such that $\tilde\varphi(x)=\alpha\tilde \psi(u^* xu)$ for all $x\in \cN$
	by \cite[Theorem 4.7]{Con74}.
It follows that for any nonzero projection $p\in N^\varphi$, one has $u^*(e\ot p) u\in \cN^{\tilde \psi}=\B(\ell^2\N)\ovt N^\psi$,
	where $e\in \B(\ell^2\N)$ denotes a rank-one projection.
By factoriality of $N^\psi$, we may find a unitary $v\in \B(\ell^2\N)\ovt N^\psi$
	such that $(uv)^*(e\ot p) uv=e\ot q$,
	where $q\in N^\psi$ is a projection with $\alpha\psi(q)=\varphi(p)$.

Consider the isomorphism $\theta: pMp=(e\ot p)\cM (e\ot p)\ni (e\ot p)x(e\ot p)\mapsto \Ad(uv)((e\ot p)x (e\ot p))\in (e\ot q)\cM (e\ot q)=qMq$
	and note
	$\varphi(x)= \alpha \tilde \psi( (e\ot q)\theta(x)(e\ot q))=\alpha\psi(q \theta(x)q)$ for any $x\in pNp$
	and it follows that $\theta: (pNp, \varphi_p)\mapsto (qNq, \psi_q)$ is a state preserving isomorphism.

Furthermore, notice that $\theta \circ E_{pNp}^{\varphi_p} = E_{pNp}^{\psi_p}\circ \theta$ as $u,v\in \cN$,
	which implies that $\varphi_p(x)=\varphi_p(E_{pNp}^{\varphi_p}(x))=\psi_q(\theta\circ E_{pNp}^{\varphi_p}(x))=\psi_q(\theta(x))$ 
	for all $x\in pMp$. 
\end{proof}

\begin{proof}[Proof of Theorem~\ref{thm: non-iso inclusions}]
Set $d=\dim(\sH_\R)$ and $(\sK_\R, V):=(\sH_\R, U)\oplus(\sK_\R', V')$, where the representation $(\sK_\R', V')$ will be determine later.
Put $N:=\Gamma_q(\sH_\R, U)''\subset M:=\Gamma_q(\sK_\R, V)''$ and denote by $\chi$ the $q$-quasi free state.
By Lemma~\ref{lem: wk coarse lower bound w/o R action}, 
	we have the $N$-$N$ bimodule $L^2(M\ominus N)^{\ot_N^k}$ is weakly coarse if $k>\frac{-\log(d)}{\log(|q|)}$.

We claim that it suffices to find a subrepresentation $\sL_\R\subset \sK_\R$ such that 
	$L^2(M\ominus B,\chi)^{\ot_{B}^k}$ is not weakly contained in $L^2B\ot L^2B$,
	where $B=\Gamma_q(\sL_\R, V)''$.
	
Indeed, if there were an isomorphism $\alpha$ between the inclusions $N\subset M$ and $B\subset M$,
	then putting $\varphi=\chi\circ\alpha^{-1}: M\to \C$,
	we have a state preserving isomorphism $\alpha: (M,\chi)\to (M,\varphi)$ that 
	restricts to an isomorphism $\alpha: (N,\chi)\to (B,\varphi)$,
	and $B\subset M$ admits a $\varphi$-preserving expectation.
It then follows that ${_{\alpha(N)}}{L^2(M\ominus B, \varphi)^{\ot_B^k}}{_{\alpha(N)}}$
	coincides with ${_N}{L^2(M\ominus N,\chi)^{\ot_N^k}}{_N}$
	and hence is weakly coarse.

Note that $B$ is a full	type ${\rm III}$ factor and $B^\chi$ is a factor \cite{KSW23}.
Applying Lemma~\ref{lem: almost periodic state} one sees that $(pBp\subset pMp, \varphi_p)$ is isomorphic to $(qBq\subset qMq,\chi_q)$
	for some nonzero projection $p\in B^\varphi$, $q\in B^\chi$,
	which in turn shows that 
	${_{pBp}}{L^2(pMp\ominus pBp,\varphi_p)^{\ot_{pBp}^k}}{_{pBp}}$
	and ${_{qBq}}{L^2(qMq\ominus qBq,\chi_q)}^{\ot_{qBq}^k}{_{qBq}}$
	are isomorphic bimodules.
Note that  ${_{pBp}}{L^2(pMp\ominus pBp,\varphi_p)^{\ot_{pBp}^k}}{_{pBp}}$ is weakly coarse 
	as ${L^2(M\ominus B, \varphi)^{\ot_B^k}}$ is,
	and thus so is $L^2(qMq\ominus qBq,\chi_q)^{\ot_{qBq}^k}$ is weakly coarse as a $qBq$-$qBq$ bimodule,
	which further implies that $L^2(M\ominus B,\chi)^{\ot_{B}^k}$ is a weakly coarse $B$-$B$ bimodule.
Indeed, setting $P=qBq$ and $\cH=L^2(qMq\ominus qBq,\chi_q)$,
	one has 
	$${_B}{L^2(M\ominus B,\chi)}{_B}={_{\M_n(P)}}{\cH\ot \C^{n^2}}{_{\M_n(P)}},$$
	as $q$ is a nonzero projection in a ${\rm II}_1$ factor $B^\chi$ and by replacing $q$ with a sub-projection we may assume
	that $\chi(q)=1/n$ for some $n\in \N$.
Thus the weak coarseness of $L^2(M\ominus B,\chi)^{\ot_{B}^k}$ follows from the weak coarseness of ${_P}{\cH}{_P}$ directly
	and our claim is justified.

To to find such a subrepresentation $\sL_\R\subset \sK_\R$, we follow the argument in Lemma~\ref{lem: non-AO lower bound}.
Note that from the proof of Lemma~\ref{lem: wk coarse lower bound w/o R action},
	the $B$-$B$ bimodule $L^2(M\ominus B)^{\ot_{B}^k}$ 
	contains vector of the form $\xi=\xi_1\ot \cdots \ot \xi_k \in (\oplus_{i=1}^{k} (\sK\ominus \sL))^{\ot k}$.
It follows that for any $\eta\in \sL_\C^{\ot n}$, one has
$$
\langle\xi, W(\eta)JW(\eta)J\xi\rangle=q^{kn} \|\xi\|^2\langle\Omega, W(\eta) JW(\eta)J\Omega\rangle.
$$
Thus the proof of Lemma~\ref{lem: non-AO lower bound} shows that if $\sL_\R$ satisfies that $2q^{2k} \dim(\sL_\R)/T^2>1$,
	where $T$ is the maximal eigenvalue of the generator of $V_{\mid \sL_\R}$,
	then $\overline{B\xi B}$ is not weakly coarse as a $B$-$B$ bimodule,
	which implies that $L^2(M\ominus B)^{\ot_B^k}$ is not weakly coarse.

Therefore, we may find such a finite dimensional representation $V': \R\to \cO(\sL_\R)$ as $k$ is only determined by $\dim(\sH_\R)$ and $q$.
The representation $V': \R\to \cO(\sK_\R')$ then can be chosen to be any finite dimensional representation containing $L_\R$ as a proper subrepresentation.
\end{proof}

\bibliographystyle{amsalpha}
\bibliography{ref}

\end{document}